\documentclass{article}

\usepackage{etoolbox}

\makeatletter
\patchcmd{\maketitle}{\@fnsymbol}{\@arabic}{}{}
\makeatother

\title{Quantile-RK and Double Quantile-RK Error Horizon Analysis}
\author{Emeric Battaglia\footnote{\url{ebattagl@uci.edu}, corresponding author}, Anna Ma\\\\ \textit{Department of Mathematics, University of California, Irvine, CA 92697 USA}}

\date{\today}

\usepackage{lipsum}
\usepackage{amsfonts}
\usepackage{graphicx}
\usepackage{epstopdf}
\usepackage{algorithmic}
\usepackage[lmargin=1in,rmargin=1in,tmargin=1in,bmargin=1in]{geometry}
\usepackage[section]{placeins}
\usepackage[font=small]{caption}
\usepackage[T1]{fontenc}
\usepackage{
    amsmath,           
    amsthm,
    mathtools,
    amssymb,           
    enumerate,		   
    graphicx,          
    lastpage,          
    multicol,          
    multirow,          
    diagbox,
    mathrsfs,
    thmtools,
    thm-restate,
    scrextend,
    algorithm,
    subcaption,
    fancyhdr,
    bbm,
    tikz,
    pgfplots,
    pgfplotstable,
    textcomp,
    charter,
    ulem,
    changepage,
    stmaryrd,
    hyperref,
    xurl,
    etoolbox,
    cleveref
}
\usepackage{amsopn}
\usepgfplotslibrary{groupplots}

\newtheorem{theorem}{\textit{Theorem}}

\newtheorem{lem}{\textit{Lemma}}

\newtheorem{cor}{\textit{Corollary}}

\newenvironment{fakeproof}[1][``\textit{Proof}'']
{\proof[#1]}
{\endproof}

\newcommand{\groupp}[1]{\l(#1\r)}
\newcommand{\groupb}[1]{\l[#1\r]}
\newcommand{\groupc}[1]{\l\{#1\r\}}

\renewcommand{\l}{\left}
\renewcommand{\r}{\right}

\newcommand{\E}{\mathbb{E}}
\newcommand{\N}{\mathbb{N}}
\newcommand{\PP}{\mathbb{P}}

\newcommand{\R}{\mathbb{R}}
\newcommand{\x}{{x^\ast}}

\newcommand{\s}{\sigma}
\newcommand{\vp}{\varphi}
\newcommand{\ve}{\varepsilon}

\newcommand{\sset}{\subseteq}

\newcommand{\sm}{\setminus}

\newcommand{\dl}{^*}

\newcommand{\ra}{\rightarrow}
\renewcommand{\lll}{\langle}
\newcommand{\rrr}{\rangle}

\renewcommand{\leq}{\leqslant}
\renewcommand{\geq}{\geqslant}
\newcommand{\Def}{\vcentcolon =}

\DeclareMathOperator*{\argmin}{arg\,min}
\DeclareMathOperator{\quant}{-quant}
\DeclareMathOperator{\supp}{supp}
\DeclareMathOperator{\unif}{Unif}
\DeclareMathOperator{\TP}{TP}
\DeclareMathOperator{\FP}{FP}
\DeclareMathOperator{\TN}{TN}
\DeclareMathOperator{\FN}{FN}

\begin{document}

\maketitle

\begin{abstract}
In solving linear systems of equations of the form $Ax=b$, corruptions present in $b$ affect stochastic iterative algorithms' ability to reach the true solution $\x$ to the uncorrupted linear system. The randomized Kaczmarz method converges in expectation to $\x$ up to an error horizon dependent on the conditioning of $A$ and the supremum norm of the corruption in $b$. To avoid this error horizon in the sparse corruption setting, previous works have proposed quantile-based adaptations that make iterative methods robust. Our work first establishes a new convergence rate for the quantile-based random Kaczmarz (qRK) and double quantile-based random Kaczmarz (dqRK) methods, which, under certain conditions, improves upon known bounds. We further consider the more practical setting in which the vector $b$ includes both non-sparse ``noise" and sparse ``corruption". Error horizon bounds for qRK and dqRK are derived and shown to produce a smaller error horizon compared to their non-quantile-based counterparts, further demonstrating the advantages of quantile-based methods. 
\end{abstract}

\begin{adjustwidth}{3em}{3em}
\end{adjustwidth}

\section{Introduction}
Finding a solution $\x$ to a \textit{consistent} linear system of equations 
    \begin{equation}
        Ax=b, \label{eq:syseq}
    \end{equation}
    where $A \in \mathbb{R}^{m \times n}$ and $m \geq n$ is a critical subroutine in many modern applications of scientific computing and mathematical data science. A low-memory footprint iterative approach for solving~\eqref {eq:syseq} is called the \textit{Kaczmarz method}~\cite{kacz}. The Kaczmarz method produces iterates $x_k$ by projecting $x_{k-1}$ onto the solution hyperplane $\{x:\lll x,a_i\rrr =b_i\}$, where $a_i$ is the $i$-th row of $A$ and $i$ is selected deterministically.  
    The most notable alteration to this algorithm, which is known as the Randomized Kaczmarz algorithm (RK), was proposed by Strohmer and Vershynin in \cite{RK}; the authors proved that when the index $i$ is selected randomly, with probability proportional to $\|a_i\|^2$,  and~\eqref{eq:syseq} is consistent, the expected error $\|x_k-\x\|^2$ diminishes to $0$ exponentially. 

    When the system~\eqref{eq:syseq} is not consistent, the least-squares solution is typically desired. However, because iterates of the RK algorithm live in the row space of $A$, and the least-squares solution need not, RK is not expected to converge to the least squares solution of an inconsistent system~\cite{needell_LS}. To address this, Zouzias and Freris \cite{Zouzias_REK} devised and provided convergence guarantees of a variant of RK, called the randomized extended Kazmarz algorithm (REK), which, in addition to performing the standard RK iterations, uses columns of the matrix $A$ to approximate the least squares solution. Other variants of RK, generally building on REK, have been proposed to meet specific regularization criteria, such as approximating a sparse least-squares solution \cite{sparseREK}.

    The system \eqref{eq:syseq} may be inconsistent due to the given vector $b$ being corrupt. To model this as an additive error, we decompose $b$ into $b=b_t +\ve$ where $b_t$ is the ``true'' value, i.e., $b_t = A \x$ and $\ve$ is the additive corruption. As the underlying system 
    \begin{equation}
        Ax= b_t, \label{eq:truesyseq}
    \end{equation}
    has a solution $\x$, it would be wishful to think iterates $x_k$ generated by RK can approximate $\x$ to any desired precision. Since only $b$ is given, and not $b_t$, it has been shown that $x_k$ can approximate $\x$ up to an error horizon \cite{needell_LS} when $A$ has full rank. For specific types of corruption, statistical methods have proven useful. For example, in the presence of possibly adversarially distributed corruption, using the mode of the residuals in each iteration can be employed to attain convergence results \cite{adversarialcorruption}. If, instead, the corruption is sparse and unbounded, using a quantile-based variation of RK allows for convergence to $\x$ \cite{qRK,qRK2,dqRK}.  

    The main contribution of this work is twofold: we first provide an alternative convergence bound for the quantile-based Kaczmarz methods, which is provably smaller than other known bounds. We then demonstrate that, under mild conditions on the noise, corruption, and conditioning of the matrix, the quantile-based variations of RK have a provably (and empirically) better convergence error horizon than guaranteed by their non-quantile-based counterpart RK~\cite{needell_LS}.
    
    \subsection{Notation}
        \label{sec:notation}
        Let $[m]=\{1,\dots, m\}$. For a vector $v \in \mathbb{R}^m$, we denote its $i^{th}$ entry by $v_i$ and we define $\supp(v)= \{i: v_i\neq 0\}$, with the following exceptions: $a_i$ which denotes the $i^{th}$ row of $A$, $b_t: = A\x$, where $\x$ is the underlying true solution, and subscript $k$ is reserved to denote the $k^{th}$ iterate, e.g., $x_k$. For any set of indices $I \subseteq [m]$, the term $A_I\in \R^{|I|\times n}$ will denote the submatrix of $A$ whose rows are given by $S$. To denote the largest and smallest non-zero singular value of a matrix $A$, we write $\s_{\max}(A)$ and $\s_{\min}(A)$, respectively. The norms used in this paper are the Frobenius norm $\|\cdot\|_F$, the $\ell_1$-norm $\|\cdot\|_1$, the Euclidean norm $\|\cdot\|$, and the sup norm $\|\cdot \|_\infty$. Though not technically a norm, we will use $\|\cdot\|_0$ to denote the $\ell_0$-norm, whose output is the number of nonzero entries in the input. 
        
        For any finite multiset $S$, we let $q\quant(S)$ denote the $q$-th quantile of $S$. We assume that $q|S|$ is always integer-valued. Because $S$ may be a multiset, we also adopt the convention that $L_q \Def \{s\in S: s\leq q\quant(S)\}$, i.e., the set containing the first $q|S|$ elements $s\in S$ satisfying $s\leq q\quant(S)$. Similarly, we define $\{s\in S: q\quant(S)<s\}\Def S\sm L_q$, where elements of $L_q$ are removed from $S$ according to multiplicity. 
        
        Another quantity of importance is 
        \begin{equation}
            \s_{q,\min}(A)=\min_{\substack{I\sset [m] \\ |I|=q m}}\inf_{\|x\|=1}\|A_I x\|,
        \end{equation}
        where $q\in [0,1]$ with $qm \in \N$. This term serves to bound $\|A_Ix\|\geq \|x\|\s_{q,\min}(A)$ whenever $I\sset [m]$ with $|I|\geq qm$. 
        
        In this work, we consider the system:
        \begin{equation}
            Ax = b_t+\eta+\xi, \label{eq:genlinsys}
        \end{equation}
        where $A\in \R^{m\times n}$, $m \geq n$ and $A$ is of full rank and $b = b_t+\eta+\xi \in \mathbb{R}^{m}$ are given. Here, we assume $Ax=b_t$ is consistent and has solution $\x$. The terms $\eta$ and $\xi$ are corruption introduced to $b_t$. To distinguish between $\eta$ and $\xi$, it is assumed $\xi$ is sparse; $\beta$ will serve as the parameter characterizing this sparsity, i.e. $\|\xi\|_0\leq \beta m$ for $\beta\in[0,1]$. Though there is no unique decomposition of the corruption into $\eta+\xi$, the intuition for the results of the paper will best follow from considering $\eta$ to be small, bounded noise and $\xi$ being unbounded, sparse corruption.        
    
    \subsection{Background}
        RK is a stochastic iterative algorithm designed to approximate the solution $\x$ to the system \eqref{eq:genlinsys} when there is no corruption present in $b$, i.e. $\eta+\xi = 0$ \cite{RK}. The iterates $x_k$ are generated by projecting the previous iterate onto a solution hyperplane $H_i=\{x:\lll x, a_i\rrr = b_i\}$ where $i$ is randomly chosen with probability proportional to $\|a_i\|^2$, that is,
        \begin{equation}
            x_{k+1}=x_k+\frac{b_i-\lll x_k,a_i\rrr}{\|a_i\|^2}a_i. \label{eq:iterate_update}
        \end{equation}
        When $m \geq n$, $A$ is of full rank, and the system is consistent, Strohmer and Vershynin showed in \cite{RK} that the approximates $x_k$ converge to $\x$ with the expected error bounded by: 
        \begin{equation}
            \E\|x_k-\x\|^2\leq \groupp{1-\frac{\s_{\min}^2(A)}{\|A\|_F^2}}^k\|x_0-\x\|^2.
        \end{equation}
        When corruption is present and the system is inconsistent, RK fails to procure iterates that can approximate $\x$, or even $x_{LS} = A^\dagger b$, with arbitrary accuracy. Without any modification to the RK algorithm, convergence is only guaranteed up to an error horizon given by Needell \cite{needell_LS}:

        \begin{theorem}[{\cite[Theorem 1]{needell_LS}}]
            The iterates of RK applied to the system \eqref{eq:genlinsys} have an expected decay given by
            \[
                \E\|x_k-\x\|^2 \leq \groupp{1-\frac{\s_{\min}^2}{\|A\|_F^2}}^k \|x_0-\x\|^2+\frac{\|A\|_F^2}{\s_{\min}^2}\max_{j}\frac{|\eta_j+\xi_j|^2}{\|a_j\|^2}.
            \]
            \label{theorem:EH}
        \end{theorem}
        \Cref{theorem:EH} asserts that, in expectation, the iterates of RK converge to a ball centered at $\x$ with radius $\frac{\|A\|_F^2}{\s_{\min}^2}\max_{j}\frac{|\eta_j+\xi_j|^2}{\|a_j\|^2}$, which we refer to as the \textit{error horizon}. 
        One may wish to ``push past'' the error horizon and converge to the least-squares solution $x_{LS}$ of \eqref{eq:genlinsys}. This is accomplished by other works such as REK \cite{Zouzias_REK} and the randomized block Kacmzarz \cite{block_LS}. Such algorithms, in principle, work by approximating $\x$ by the iterates $x_k$ using approximations of $b$'s projection in the column space of $A$. 
        
        When the system is inconsistent due to \textit{sparse}, large corruptions, the least squares solutions $x_{LS} = A^\dagger b$ may not generally be a good approximation of $\x$. Due to the sparsity of the noise, however, not all is lost, and $\x$ can be recovered. In particular, the quantile-based randomized Kaczmarz algorithm (qRK) \cite{qRK,qRK2}, and the double quantile-based randomized Kaczmarz algorithm (dqRK) \cite{dqRK} may still converge to $\x$. These algorithms operate, in principle, by only allowing specific solution spaces to be admissible at iteration $k$. Under the assumption that $A$ is row-normalized, entries of the residual at iteration $k$: 
        \[
            \groupc{\lll x_k, a_j\rrr -b_j}_{j=1}^m,
        \]
        are used to determine whether an index can or cannot be selected. More precisely, qRK and dqRK only randomly select indices from the set 
        \begin{equation}
            \groupc{j\in [m]:|\lll x_k, a_j\rrr -b_j|\leq Q}\quad\text{and}\quad\groupc{j\in [m]: Q_0<|\lll x_k, a_j\rrr -b_j|\leq Q},
        \label{eq:admiss_set}
        \end{equation}
        at each iteration, respectively, where 
        \[
            Q_0= q_0\quant\groupp{\groupc{|\lll x_k, a_j\rrr -b_j|}_{j=1}^m} \quad \text{and} \quad Q= q\quant\groupp{\groupc{\lll x_k, a_j\rrr -b_j|}_{j=1}^m}.
        \]
        When $A$ is not row-normalized, the $j$-th residual entry is scaled by a factor of $\frac{1}{\|a_j\|^2}$. For generality, qRK and dqRK are outlined in \Cref{alg:qrk} and \Cref{alg:dqrk}, respectively, without the assumption that $A$ is row-normalized. 
        \begin{figure}[t!]
            \centering
            \begin{minipage}[t]{.48\textwidth}
                \centering
                \begin{algorithm}[H]
                    \caption{Quantile-Based Randomized Kaczmarz Method (qRK)}
                    \begin{algorithmic}
                        \STATE \textbf{Inputs: } $A,\, b,\, q,\, x_0,\, K$
                        \STATE $k \gets 0$
                        \FOR{$k<K$}
                            \STATE $r \gets b-Ax_k$
                            \STATE $Q \gets q\quant\{|r_j|\}_{j=1}^m$ 
                            \STATE $I \gets \groupc{j\in [m]:|r_j|\leq Q }$
                            \STATE Select $i\in I$ with probability $\frac{\|a_i\|^2}{\|A_I\|_F^2}$ 
                            \STATE $x_{k+1} \gets x_k + \frac{r_i}{\|a_i\|}a_i$
                            \STATE $k\gets k+1$
                        \ENDFOR
                    \end{algorithmic}
                    \label{alg:qrk}
                \end{algorithm} 
            \end{minipage}
            \begin{minipage}{0.02\textwidth}
            \[\,\]    
            \end{minipage}
            \begin{minipage}[t]{0.48\textwidth}
                \centering
                \begin{algorithm}[H]
                    \caption{Double Quantile-Based Random Kaczmarz Method (dqRK)}
                    \begin{algorithmic}
                        \STATE \textbf{Inputs: } $A,\, b,\, q_0,\, q,\, x_0,\, K$
                        \STATE $k \gets 0$
                        \FOR{$k<K$}
                            \STATE $r \gets b-Ax_k$
                            \STATE $Q_0, Q \gets q_0,q\quant\{|r_j|\}_{j=1}^m$
                            \STATE $I \gets \groupc{j\in [m]: Q_0< |r_j|\leq Q}$
                            \STATE Select $i\in I$ with probability $\frac{\|a_i\|^2}{\|A_{I}\|_F^2}$
                            \STATE $x_{k+1} \gets x_k + \frac{r_i}{\|a_i\|}a_i$
                            \STATE $k\gets k+1$
                        \ENDFOR
                    \end{algorithmic}
                    \label{alg:dqrk}
                \end{algorithm}
            \end{minipage}
        \end{figure}
        To understand heuristically why qRK and dqRK might converge to $\x$ even in the presence of only sparse corruption, only one observation is key: it should stand to reason, at least when $x_k$ is sufficiently close to $\x$ and $\eta=0$, that the collection of indices corresponding to the larger residuals $\frac{|\lll x_k, a_j\rrr -b_j|}{\|a_j\|}$ contains the support of the sparse-corruption $\xi$. Because $(1-q)m$ indices are removed from consideration using the upper quantile $Q$, $q$ should be chosen such that $1-q > \beta$, where $\|\xi \|_0\leq \beta m$. Then, by largely ignoring indices that likely correspond to corrupted entries, qRK and dqRK should be able to (mostly) use non-corrupt information to approximate $\x$.  

        Both qRK and dqRK decrease the likelihood of selecting indices of rows that are too violated, as these likely correspond to corrupted entries in $b$. However, since rows are removed from consideration at each iteration, qRK and dqRK require matrices to have enough redundant information encapsulated in their rows to produce $\x$. 
        \Cref{theorem:qrk} and \Cref{theorem:dqrk} provide convergence results for qRK and dqRK, respectively.
        \begin{theorem}[qRK, {\cite[Main Theorem]{qRK}}]
            Suppose $A$ is row-normalized and full-rank. Let $\x$ be the solution to the over-determined system $Ax=b_t$ and $b=b_t+\xi$ where $\|\xi\|_0\leq \beta m$. Suppose $\beta < q<1-\beta$, and 
            \begin{equation}
                \frac{q}{q-\beta}\groupp{\frac{2\sqrt{\beta}}{\sqrt{1-q-\beta}}+\frac{\beta}{1-q-\beta}}<\frac{\s^2_{q-\beta,\min}(A)}{\s_{\max}^2(A)}. \label{eq:qRK_condition}
            \end{equation}
            Then the expected error of the qRK algorithm decays as
            \[
                \E\|x_k-\x\|^2\leq (1-C)^k\|x_0-\x\|^2,
            \]
            where
            \begin{equation}
                C=(q-\beta)\frac{\s^2_{q-\beta,\min}(A)}{q^2m}-\frac{\s^2_{\max}(A)}{qm}\groupp{\frac{2\sqrt{\beta}}{\sqrt{1-q-\beta}}+\frac{\beta}{1-q-\beta}}. \label{eq:qRK_constant}
            \end{equation}
            \label{theorem:qrk}
        \end{theorem}
        \begin{theorem}[dqRK, {\cite[Theorem 2]{dqRK}}]
            Suppose $A$ is row-normalized and full-rank. Let $\x$ be the solution to the over-determined system $Ax=b_t$ and $b=b_t+\xi$ where $\|\xi\|_0\leq \beta m$. Suppose $\beta < q_0 < q < 1-\beta$, $q-q_0>\beta$, and 
            \[
                \frac{q}{q-q_0-\beta}\groupp{\frac{2\sqrt{\beta}}{\sqrt{1-q-\beta}}+\frac{\beta}{1-q-\beta}} < \frac{1}{\s_{\max}^2(A)}\groupp{\s_{q-\beta,\min}^2(A)+\frac{\s_{q_0-\beta,\min}^2(A)}{q_0m}}.
            \]
            If $x_0$ is chosen such that $\langle x_0, a_i\rangle = b_i$ for some $i\in [m]$, then the expected error of the dqRK algorithm decays as
            \[
                \E\|x_{k}-\x\|^2\leq (1-C)^k \|x_0-\x\|^2,
            \]
            where 
            \[
                C = \groupp{q-q_0-\beta}\groupp{\frac{\s_{q-\beta,\min}^2(A)}{(q-q_0)q m}+\frac{\s_{q_0-\beta,\min}^2(A)}{(q-q_0)q_0q m^2}} - \frac{\s_{\max}^2(A)}{(q-q_0)m}\groupp{\frac{2\sqrt{\beta}}{\sqrt{1-q-\beta}}+\frac{\beta}{1-q-\beta}}.
            \]
            \label{theorem:dqrk}
        \end{theorem}

\section{Contributions}
    \label{sec:contributions}

    The contributions of this paper are twofold. First, we provide an alternative proof of the convergence rate for qRK and show that under mild constraints, this rate is tighter than previously established results \cite{qRK}. Secondly, we show qRK, as outlined by \Cref{alg:qrk}, has error horizon dependent on the conditioning of $A$ and $\|\eta\|_\infty$, as opposed to $\|\eta\|_1$ found in \cite{timevaryingcorruptionqRK}. These results are also naturally extended to dqRK.

    The rest of the paper is structured as follows: \Cref{sec:related_work} presents closely related works that consider stochastic iterative approaches for solving ~\eqref{eq:genlinsys}. \Cref{sec:main_results} contains the main results of this paper. \Cref{sec:qRK_new_bounds} presents the new convergence rate for qRK along with a discussion of the setting under which the new rate serves as an improvement. \Cref{sec:qRK_EH} provides our error horizon bound attained by both qRK compares it to a previously known error horizon bound given by \cite{timevaryingcorruptionqRK}. Following the presentation of the results for qRK, \Cref{sec:dqRK_analogous_results} introduces the analogous extension of this work to dqRK. Finally, \Cref{sec:empirical_results} provides empirical results to reinforce the results presented in previous sections. The proofs of the statements in \Cref{sec:main_results} may be found in \Cref{sec:proofs}.

\section{Related Work}
    \label{sec:related_work}
    Previous works have also considered a decomposition of the right hand side noise into the sum of a sparse, unbounded term $\xi$ and a noise term $\eta$. The decomposition of error into noise and sparse corruption is not unique to the RK family of algorithms; indeed, some intuition and results surrounding qRK with this error extend to Min-$k$ Loss stochastic gradient descent (MKL-SGD), a more robust variation of SGD \cite{lowestlossSGD}. Moreover, both forms of corruption $\eta$ and $\xi$ may be time-dependent \cite{timevaryingcorruptionqRK}. This section provides context for our contributions with respect to two closely related works.    
    
    The RK method can be viewed as a special case of SGD. Indeed, as outlined in \cite{SGD_RK}, we wish to solve a system $Ax=b$ by minimizing 
    \[
        F(x)\Def\frac{1}{2}\|Ax-b\|^2, 
    \]
    in which case defining $f_i(x)\Def \frac{m}{2}(\lll a_i,x\rrr -b_i)^2$ yields $\nabla F(x)=\E \nabla f_i(x)$ when $i$ is chosen uniformly randomly from $[m]$. Applying SGD with this objective function after re-weighting the uniform distribution over $[m]$ so that index $i$ is chosen with probability proportional to $\|a_i\|^2$ yields iterates of the form 
    \[
        x_{k+1}= x_k+c\frac{b_i-\lll a_i,x_k\rrr}{\|a_i\|^2}a_i,
    \]
    where $c$ is the step size; if $b=b_t+\eta + \xi$ and $\x$ satisfies $A\x=b_t$, then whenever $c<1$,  
    \[
        \E\|x_k-\x\|^2\leq \groupp{1-\frac{2c(1-c)\s_{\min}^2(A)}{\|A\|_F^2}}^k\|x_0-\x\|^2+\frac{c}{1-c}\frac{\s_{\min}^2(A)}{a_{\min}^2}\sum_{j\in[m]} \|a_j\|^2|(\eta+\xi)_j|^2,
    \]
    where $a_{\min}=\min_{j\in [m]}\|a_j\|$ \cite{SGD_RK}.
    
    The quantile-based variants of the RK methods can also be adapted to the SGD setting. In particular, extending qRK's principle of excluding indices corresponding to large residuals has a similar counterpart in literature. As defined in \cite{lowestlossSGD}, if the loss function $F(x)$ may be expressed as $F(x)=\frac{1}{m}\sum_{i=1}^m  f_i(x)$, then MKL-SGD is designed such that at each iteration $t$, a set $S$ of $k$ samples is randomly chosen, then the index $i\Def \argmin_{j\in S} f_j(x_t)$ corresponding to the lowest loss within $S$ is selected to construct the next iterate $x_{t+1}=x_t-c \nabla f_i(x_t)$. This closely mimics the selection strategy of SKM with $\beta m = k$ outlined in \cite{greedworks}, with the exception that the lowest loss, in place of the largest, is selected. By avoiding large loss, just as qRK avoids large residuals, MKL-SGD is less susceptible to corruptions than vanilla SGD. If $k$ samples are chosen at each iteration, then it is clear that whenever index $i$ is chosen and $q=\frac{m-k+1}{m}$, then 
    \[
        f_i(x_t) \leq q\quant\groupp{\groupc{f_j(x_t)}_{j=1}^m}.
    \]
    That is, the index $i$ may only be selected if it corresponds to a loss less than a given quantile. Then, just as RK can be framed in the SGD setting, the qRK method also has an SGD-based counterpart that closely mimics its main feature: robustness. As shown in this paper for qRK (and dqRK) compared to vanilla RK, MKL-SGD may have an error horizon smaller than that of its vanilla counterpart, SGD.

    The qRK algorithm was originally designed when the corruption in $b$ is sparse. When it is not, the best we might hope to attain is a bound dependent on the error horizon. This may also occur when the corruption is not sparse enough and based on the choice of $q$. The error horizon bound for a particular class of random matrices has already been studied by Jarman and Needell \cite{qRK_EH_subgaussian}. We will focus on providing a similar bound without assumptions on the choice of the matrix beyond its conditioning. The error horizon bound in the more general setting with corruption varying with time has previously been studied in \cite{timevaryingcorruptionqRK}. The most important distinction between the error horizon bound (in the time-independent noise setting) in this work and that of \cite{timevaryingcorruptionqRK} is that, ignoring the conditioning of $A$, ours is a function of $\|\eta\|_\infty$ and that the previous one is a function of $\|\eta\|_1$. In \Cref{sec:error_horizon_bound_comparison}, we directly compare our new error horizon bound to that of \Cref{theorem:timevar_theorem} and \Cref{cor:timevar_cor}. The error horizon bound has also been studied for the quantile randomized sparse Kaczmarz method (qRaSK), a variant of qRK designed to recover a sparse solution, by L. Zhang, H. Zhang, and H. Wang \cite{sparse_qRK_EH}. By setting the $\lambda$ parameter to $0$ and utilizing the inexact step $t_k = \lll a_i, x_k\rrr -b_i$, qRaSK simplifies to qRK. However, the associated error horizon bound for one iteration recovered in the qRK setting is dependent on the largest singular value and dimensions of $A$. In contrast, ours only depends on $\eta$ and the parameters $\beta$ and $q$. More detail can be found in \Cref{sec:error_horizon_bound_comparison}.

\section{Main results}
\label{sec:main_results}    
Before presenting our main results, we introduce notation that will help present the convergence bounds more intuitively. For row-normalized $A$, we denote $\kappa_q(A)=\frac{\s_{\max}(A)}{\s_{q-\beta, \min}(A)}$ and $~{\hat\kappa_q(A)=\frac{\sqrt{qm}}{\s_{q-\beta,\min}(A)}}$, where $\beta$ is the fixed sparsity parameter for qRK and dqRK, i.e. $\|\xi\|_0\leq \beta m$. Intuitively, $\kappa_q(A)$ measures how well-conditioned the collection of subsets of $A$ of size $(q-\beta)m$ is, and $\hat\kappa_q(A)$ quantifies the worst-case scaled condition number of any submatrix of $A$ with $(q-\beta)m$ rows. This measure of conditioning of the subsystems addresses how robust qRK and dqRK may be in the presence of $\beta$-sparse corruption, but does not provide information about how qRK and dqRK fair in the presence of general noise $\eta$. Then, $\kappa_q(A)$ and $\hat\kappa_q(A)$ do not depend on the presence of the noise $\eta$.

At each iteration of qRK when $\eta = 0$, instead of randomly selecting from any row of $A$, rows are selected from a subset of ``admissible'' rows \eqref{eq:admiss_set}. Ideally, only ``non-corrupt'' rows, or rows such that $\xi_i = 0$, are positively identified and appear in these admissible sets. Thus, the accuracy of this set plays a role in convergence. Suppose the set of admissible indices is considered a predicted positive, and the set of non-corrupt indices is considered true positives. In that case, we may express the number of false positives (admissible corrupt indices) as $\FP$, the number of false negatives (non-admissible non-corrupt indices) as $\FN$, and the number of true negatives (non-admissible corrupt indices) as $\TN$. This relationship is illustrated in the confusion matrix in \Cref{tab:confusion_mat}.
\begin{table}[!h]
    \centering
    \begin{tabular}{l|l|c|c|c}
        \multicolumn{2}{c}{}&\multicolumn{2}{c}{Predicted}&\\
        \cline{3-4}
        \multicolumn{2}{c|}{}&Admissible&Not Admissible&\multicolumn{1}{c}{Total}\\
        \cline{2-4}
        \multirow{2}{*}{Actual}& Not Corrupt & $\TP$ & $\FN$ & $(1-\beta)m$\\
        \cline{2-4}
        & Corrupt & $\FP$ & $\TN$ & $\beta m$\\
        \cline{2-4}
        \multicolumn{1}{c}{} & \multicolumn{1}{c}{Total} & \multicolumn{1}{c}{$qm$} & \multicolumn{1}{c}{$(1-q)m$} & \multicolumn{1}{c}{$m$}\\
    \end{tabular}
    \caption{Confusion matrix for classification using quantile information of residuals. For the purposes of the totals, it is assumed $\|\xi\|_0=\beta m$.}
    \label{tab:confusion_mat}
\end{table}

We let $p$ serve as a worst-case measure of what proportion of the admissible rows are guaranteed to be non-corrupt. For qRK, if all corrupt rows appear in the admissible set, then we may only ensure $(q-\beta)m$ of the $qm$ rows are admissible and non-corrupt, so $p=\frac{q-\beta}{q}$. Of course, we wish $p$ to be as close to 1 as possible. Our bounds also depend on the term $r\Def \beta / (1-q-\beta)$. Two interpretations for $r$ help shed light on \Cref{theorem:qrk}. Since $\FP, \TN \leq \beta m$, and $\FN \geq (1-q-\beta)m$, the ratio of false positives to false negatives $\frac{\FP}{\FN}$ is bounded above by $r$. Increasing the minimum number of false negatives and decreasing the maximum number of false positives are both accomplished by decreasing the number of entries allowed to be corrupt. We may also consider interpreting $\beta/(1-q-\beta)$ as a bound for $\FP/(\FP + \FN)$, i.e. the ratio of false positives to all the erroneous classifications. It stands to reason that an admissible corrupt index ($\FP$) will likely do more harm per iteration than not admitting a non-corrupt index ($\FN$), so just as \Cref{theorem:qrk} and \Cref{theorem:dqrk} would suggest, it is favorable to ensure $\beta/(1-q-\beta)$ is small, in turn guaranteeing $\FP/(\FP + \FN)$ is small. Using this new notation, the hypothesis \eqref{eq:qRK_condition} of \cref{theorem:qrk} may be succinctly rewritten using
\[
    \frac{1}{p}\groupp{2\sqrt{r}+r}< \kappa_q^{-2}(A),
\]
and the constant in \eqref{eq:qRK_constant} becomes
\[
    C=p\hat\kappa_q^{-2}(A)-\frac{\s_{\max}^2(A)}{qm}\groupp{2\sqrt{r}+r}.
\]
Now, we can more clearly see how $\frac{\s_{\max}^2(A)}{qm}\groupp{2\sqrt{r}+r}$ acts as a measure of how adversarial the corruption can be, due either to the ratio of $\beta$ and $q$ or the conditioning of $A$. For example, when $\kappa_q(A)$ is smaller, the effect of selecting a corrupted row is less impactful, which would be reflected in having a smaller decay factor $1-C$. Alternatively, if $\beta$ is reduced, then it is less likely a corrupted row will be selected, so we should expect faster convergence, on average; this would be reflected with a smaller $\frac{\s_{\max}^2(A)}{qm}\groupp{2\sqrt{r}+r}$ and a larger $p$. 

\subsection{Alternate qRK Bounds}
\label{sec:qRK_new_bounds}
\Cref{cor:qRK_STD_NEWbound} presents the new convergence bounds of qRK. The change of convergence rate and criteria in \Cref{cor:qRK_STD_NEWbound} stem from improving the factor $2 \sqrt{r} + r$ to $\beta m /\s_{\max^2} + 2r$. This is done by controlling the magnitude of the adversarial effects of corruption.
    \begin{cor}
        (Alternative qRK Bound). Suppose $\beta < q <1-\beta$. Assume $A$ is row-normalized and full-rank. Let $\x$ be the solution to the over-determined system $Ax=b_t$, $\|\xi\|_0\leq \beta m$, and $\eta=0$. If
        \[
            \frac{1}{p}\groupp{\frac{\beta m}{\s_{\max}^2(A)}+2r}< \kappa_q^{-2},
        \]
        then the $k$-th iterate of qRK applied to the system \eqref{eq:genlinsys} witnesses
        \[
            \E\|x_k-\x\|^2\leq (1-C)^k\|x_0-\x\|^2,
        \]
        where
        \[
            C=p\hat\kappa_q^{-2}(A)-\frac{1}{q}\groupp{\beta+\frac{2\s_{\max}^2(A)r}{m}}.
        \]
                \label{cor:qRK_STD_NEWbound}

    \end{cor}

    \Cref{theorem:boundcomparisontheorem} addresses the setting in which the convergence rate of \Cref{cor:qRK_STD_NEWbound} is faster than that of \Cref{theorem:qrk}. The conditions to be met are relatively lax, and sometimes can be met independently of the size of $m$. This is discussed in more detail after the proof of \Cref{theorem:boundcomparisontheorem}.
    \begin{theorem}
        \label{theorem:boundcomparisontheorem}
        If $r<4$ and $\s_{\max}^2(A)>\beta m \groupp{\frac{1-q-\beta}{2\sqrt{\beta}\sqrt{1-q-\beta}-\beta}}$, then the decay factors 
        \begin{align*}
            \alpha_1^{\text{qRK}} &\Def 1-p\hat\kappa_q^{-2}(A)+\frac{1}{q}\groupp{\beta+\frac{2\s_{\max}^2(A)r}{m}} \\
            \alpha_2^{\text{qRK}} &\Def 1 - p\hat\kappa_q^{-2}(A) + \frac{\s_{\max}^2(A)}{qm}\groupp{2\sqrt{r}+r},
        \end{align*}
        given by \Cref{cor:qRK_STD_NEWbound} and \Cref{theorem:qrk}, respectively, witness $\alpha_1^{\text{qRK}}<\alpha_2^{\text{qRK}}$.
    \end{theorem}
    \begin{proof}
        By re-writing $\beta m \groupp{\frac{1-q-\beta}{2\sqrt{\beta}\sqrt{1-q-\beta}-\beta}}$ as 
        \begin{equation}
            \beta m\groupp{\frac{1}{\frac{2\sqrt{\beta}}{\sqrt{1-q-\beta}}-\frac{\beta}{1-q-\beta}}}, \label{eq:firstcondition_rewritten}
        \end{equation}
        it becomes clear the denominator is positive because $\frac{\beta}{1-q-\beta}<4$, whereby 
        \[
            \frac{\s_{\max}^2(A)}{qm}\groupp{\frac{2\sqrt{\beta}}{\sqrt{1-q-\beta}}+\frac{\beta}{1-q-\beta}-\frac{2\beta}{1-q-\beta}}>\frac{\beta}{q}.
        \]
        Thus 
        \[
            \frac{\beta}{q}+\frac{ 2\beta\s_{\max}^2(A)}{qm(1-q-\beta)}-\frac{\s_{\max}^2(A)}{qm}\groupp{\frac{2\sqrt{\beta}}{\sqrt{1-q-\beta}}+\frac{\beta}{1-q-\beta}}<0,
        \]
        making it clear $\alpha_1^\text{qRK}<\alpha_2^\text{qRK}$. 
    \end{proof}

    The condition that $r<4$ may be re-written as $\beta< \frac{4}{5}(1-q)$. In other words, the first condition of \Cref{theorem:boundcomparisontheorem} simply requires that the number of corrupted indices exceeds four-fifths of the indices removed from consideration at each step of qRK. When compared to the hypotheses of \Cref{cor:qRK_STD_NEWbound} and \Cref{theorem:qrk}, this is very lax. Indeed, if $\beta$ and $q$ are chosen so poorly that $r\geq 1$, i.e. $\beta\geq\frac{1}{2}(1-q)$, then the hypothesis 
    \[
        \frac{1}{p}\groupp{2\sqrt{r}+r}< \kappa_q^{-2}(A),
    \]
    from \Cref{theorem:qrk} is not satisfied because $\frac{\s_{q-\beta,\min}^2(A)}{\s_{\max}^2(A)}<1$. This is not particular to \Cref{theorem:qrk}; the hypothesis 
    \[
        \frac{1}{p}\groupp{\frac{\beta m}{\s_{\max}^2(A)}+2r}<\kappa_q^{-2}(A),
    \]
    of \Cref{cor:qRK_STD_NEWbound} may also fail under similar circumstances. Because the hypotheses of \Cref{theorem:qrk} require $r<1$, it seems it would be unnecessary to leave $r<4$ as a hypothesis in \Cref{theorem:boundcomparisontheorem}. However, qRK has displayed success empirically even in the case where $\beta$ is (relatively) large and the linear system fails the hypotheses of \Cref{cor:qRK_STD_NEWbound} or \Cref{theorem:qrk}; should a later result provide analytical results in line with these empirical results, it might be useful to keep the original hypothesis $r<4$ in \Cref{cor:qRK_STD_NEWbound}. 

    The requirement that 
    \begin{equation}
        \s_{\max}^2(A)>\beta m \groupp{\frac{1-q-\beta}{2\sqrt{\beta}\sqrt{1-q-\beta}-\beta}}, \label{eq:boundcomphypothesis}
    \end{equation} 
    is also relatively easy to satisfy. It is easy to verify $\|A\|_F\leq \s_{\max}(A)\sqrt{n}$. Because $A$ is assumed to be row normalized, this implies $\s_{\max}(A)\geq \sqrt{\frac{m}{n}}$. Then whenever 
    \[
        \frac{1}{n}>\beta  \groupp{\frac{1-q-\beta}{2\sqrt{\beta}\sqrt{1-q-\beta}-\beta}},
    \]
    the condition \eqref{eq:boundcomphypothesis} is immediately satisfied, independent of the size of $m$. Put differently, the convergence rate of qRK for matrices that are particularly tall and skinny with
    \[
        n< \frac{2\sqrt{r}-r}{\beta},
    \]
    is better characterized by $\alpha_2^{\text{qRK}}$.

\subsection{qRK Error Horizon}
\label{sec:qRK_EH}    
    This section presents a new error horizon bound for qRK along with some brief remarks on interpreting the theorem. \Cref{sec:error_horizon_bound_comparison} then compares this error horizon bound to a previously established weaker bound. Previously, we considered when the noise $\eta+\xi$ is $\beta m$ sparse. Now, we move to the setting in which $\eta+\xi$ may be dense.
    \begin{theorem}
        \label{theorem:qRK_EH}
        (qRK Error Horizon). Suppose $A$ is row-normalized and full-rank. Let $\x$ be the solution to the over-determined system $Ax=b_t$ and $b=b_t+\eta+\xi$ with $\|\xi\|_0\leq \beta m$ and $\supp(\eta)\cap\supp(\xi)=\varnothing$. Suppose further that $\beta < q<1-\beta$ and 
        \[
            \frac{1}{p}\groupp{\frac{\beta m}{\s_{\max}^2(A)}+2r+\frac{4\sqrt{\beta m}\sqrt{r}}{\s_{\max}(A)}}<\kappa_q^{-2}(A).
        \]
        Then the expected error of the qRK algorithm decays as
        \[
            \E\|x_k-\x\|^2 \leq (1-C)^k\|x_0-\x\|^2+\frac{1}{C}\groupp{2r\frac{1-q}{q}+ 1}\|\eta\|_\infty^2.
        \]
        where
        \[
            C=p\hat\kappa_q^{-2}(A)-\frac{1}{q}\groupp{\beta+\frac{2\s_{\max}^2(A)r}{m}+\frac{4\s_{\max}(A)\sqrt{\beta r}}{\sqrt{m}}}.
        \]
    \end{theorem}

    \Cref{theorem:qRK_EH} states that whenever $C>0$, we should expect the iterates of qRK should converge exponentially to the boundary of the ball centered at $\x$ with a radius given as a function of the infinity norm of the noise $\eta$, $\beta$, $q$, and the conditioning of $A$. We note that $2r\frac{1-q}{q}+1<2$ whenever the hypotheses of \Cref{theorem:qRK_EH} are satisfied and $q>1/2$. 
    
    If $\eta=0$ or $\|\eta+\xi\|_0\leq \beta m$, i.e. the full corruption is sparse, then we recover convergence as in \Cref{theorem:qrk}, but with a larger decay factor. In particular, when comparing the decay terms from \Cref{cor:qRK_STD_NEWbound} and \Cref{theorem:qRK_EH}, we observe an additional $4\s_{\max}(A)\sqrt{\beta r}/\sqrt{m}$ term. This is a result of using rows in $\supp(\eta)$ to bound the quantile $Q$, which in turn helps bound the expected error after selecting an index in $\supp(\xi)$. Because the residual entries $\lll x_k,a_i\rrr-(b_t+\eta)_i$, where $i\in \supp(\eta)$, serve as an imperfect proxy for the ideal residual entries $\lll x_k, a_i\rrr-(b_t)_i$, we might expect a lower performance from qRK compared when $\|\eta+\xi\|_0\leq \beta m$. 

    The lack of importance of the hypothesis $\supp(\xi)\cap \supp(\eta)=\varnothing$ is trivial to show. Suppose the remaining hypotheses of \Cref{theorem:qRK_EH} are satisfied for some choice of $\xi$ and $\eta$. Suppose further that the indices $I$ in the support of $\xi$ and $\eta$ are non-empty, i.e. $I\Def \supp(\xi)\cap \supp(\eta)\neq \varnothing$. Define vectors $\hat{\xi}$ and $\hat{\eta}$ by
    \[
        \hat{\xi}_j =\begin{cases}
            \xi_j + \eta_j, & \text{if } j\in I \\
            \xi_j, & \text{if } j\not\in I,
        \end{cases}
    \]
    and 
    \[
        \hat{\eta}_j =\begin{cases}
            0, & \text{if } j\in I \\
            \eta_j, & \text{if } j\not\in I.
        \end{cases}
    \]
    Note that $b=b_t+\eta+\xi=b_t+\hat{\eta}+\hat{\xi}$, whereby applying qRK to the system $Ax=b_t+\eta+\xi$ is the same as applying it to $Ax=b_t+\hat{\eta}+\hat{\xi}$. We also see that $\|\hat{\xi}\|_0\leq \|\xi\|_0\leq \beta m$ and $\supp(\hat{\xi})\cap \supp(\hat{\eta})=\varnothing$, so 
    \begin{align*}
        \E\|x_k-\x\|^2 &\leq (1-C)^k\|x_0-\x\|^2+\frac{1}{C}\groupp{2r\frac{1-q}{q}+ 1}\|\hat{\eta}\|_\infty^2 \\
        &\leq (1-C)^k\|x_0-\x\|^2+\frac{1}{C}\groupp{2r\frac{1-q}{q}+ 1}\|\eta\|_\infty^2.
    \end{align*}

    At no point does the algorithm qRK ``see'' what choice of $\eta$ and $\xi$ was made. The proof only requires the corruption to have \textit{some} decomposition into $\eta+\xi$ satisfying the hypotheses of \Cref{theorem:qRK_EH}. As such, we may opt to write $b$ as $b_t+\ve$. It is then evident that if we let $I$ index the first $\beta m$ largest entries in magnitude of $\ve$ and let $\xi$ be given by
    \[
        \xi_j =\begin{cases}
            \ve_j, & \text{if } j\in I \\
            0, & \text{if } j\not\in I,
        \end{cases}
    \]
    and $\eta = \ve - \xi$, then $\supp(\xi)\cap \supp(\eta)=\varnothing$. Thus, the smallest error horizon depends on the $(\beta m - 1)^{th}$ smallest entry of $\epsilon = \eta+\xi$, and \Cref{cor:generalqRK_EH} can be obtained as a result of \Cref{theorem:qRK_EH}.
    \begin{cor}
        \label{cor:generalqRK_EH}
        Suppose $A$ is row-normalized and full-rank. Let $\x$ be the solution to the over-determined system $Ax=b_t$ and $b=b_t+\ve$. Suppose further that $\beta < q<1-\beta$ and 
        \[
            \frac{1}{p}\groupp{\frac{\beta m}{\s_{\max}^2(A)}+2r+\frac{4\sqrt{\beta m}\sqrt{r}}{\s_{\max}(A)}}<\kappa_q^{-2}(A).
        \]
        Then the expected error of the qRK algorithm decays as
        \[
            \E\|x_k-\x\|^2 \leq (1-C)^k\|x_0-\x\|^2+\frac{1}{C}\groupp{2r\frac{1-q}{q}+ 1}\ve_{(\beta m+1)}^2.
        \]
        where $\ve_{(j)}$ is the $j$-th largest entry in magnitude of $\ve$, and
        \[
            C=p\hat\kappa_q^{-2}(A)-\frac{1}{q}\groupp{\beta+\frac{2\s_{\max}^2(A)r}{m}+\frac{4\s_{\max}(A)\sqrt{\beta r}}{\sqrt{m}}}.
        \]
    \end{cor}

It is important to note the error horizon bound of RK from \Cref{theorem:EH} is not necessarily larger than that of qRK in \Cref{theorem:qRK_EH}. Indeed, this should be expected when the largest (in magnitude) entry of $\eta+\xi$ is on the same order as its $(\beta m + 1)$-th largest entry. If it is the case that the system $Ax=b_t+\ve$ witnesses
\begin{equation}
    \frac{\ve_{(\beta m+1)}}{\ve_{(1)}}<\groupb{\frac{C m}{\s_{\min}^2(A)\groupp{\frac{2r(1-q)}{q}+1}}}^{\frac{1}{2}},
    \label{eq:EH_condition}
\end{equation}
where $\ve_{(j)}$ is defined as in \Cref{cor:generalqRK_EH}, then the error horizon bound of qRK from \Cref{theorem:qRK_EH} is smaller than that of RK in \Cref{theorem:EH}. In other words, to guarantee qRK has an error horizon bound smaller than RK, it is sufficient to require a large enough disparity between the $(\beta m-1)$-th largest entry and the largest entry of the corruption in the system. In the setting where sparse corruption is present, this states a significant enough gap between the magnitude of the noise and the magnitude of the corruption is required. 

The above analysis focuses on the \textit{bounds} of the error horizons of RK and qRK. However, because RK selects rows completely at random, there is a nonzero probability of selecting the row corresponding to the largest corruption in $b$ infinitely often. If we label this index $i$, then the iterate $x_{k+1}$ given by selecting index $i$ would witness
\[
    \|x_{k+1}-\x\|^2= \|x_k^\ast -\x\|^2+|\ve_i|^2\geq |\ve_i|^2,
\]
where $x_k^\ast$ is as defined in the proof of \Cref{theorem:qRK_EH} at the end of \Cref{sec:qRK_EH_proof}. Therefore, when \eqref{eq:EH_condition} holds, we are not only guaranteed that the error horizon bound of qRK is smaller than RK, but we should also expect that for a given run of qRK and RK, the realized error horizon observes this relation.

\subsubsection{Comparison to previously established error horizon bounds}
    \label{sec:error_horizon_bound_comparison}
    In~\cite{timevaryingcorruptionqRK,sparse_qRK_EH}, the authors also consider the mixture of sparse and non-sparse noise. In this section, we compare our bound to the results in the non-time-varying setting. In particular, we show our bound improves \cite{timevaryingcorruptionqRK} for sufficiently tall matrices and for a collection of lower-dimensional matrices, and we show our bound is comparable to that of \cite{sparse_qRK_EH} with the added benefit of holding independent of the choice of matrix $A$. 
    
    To more precisely outline the result presented in \cite{timevaryingcorruptionqRK}, we first define some terms. In \cite{timevaryingcorruptionqRK}, a more general setting is considered where the noise and corruption may depend on time. To express this, the right hand side of the system \eqref{eq:genlinsys} is replaced with $b^{(k)}=b_t+\eta^{(k)}+\xi^{(k)}$ where $\eta^{(k)}$ and $\xi^{(k)}$ are the time-varying noise and sparse corruption, respectively. To describe the new decay factor for qRK, let
    \[
        \vp \Def \hat\kappa^{-2}_q p-\frac{\s_{\max}^2(A)}{qm}\groupp{2rs+r^2s^2}-\frac{\s_{\max}(A)}{qm}\frac{1}{\sqrt{\beta m}}\groupp{r+r^2s},
    \] 
    where $s = \sqrt{\frac{1-\beta}{\beta}}$.
    
    To describe the error horizon bound, let 
    \[
        \zeta \Def \frac{\s_{\max}(A)}{qm}\frac{ 1}{\sqrt{\beta m}}\groupp{r+r^2s}+\frac{r^2}{q\beta m^2},
    \]
    and
    \[
        \gamma_k \Def \sum_{j=0}^k (1-\vp)^{k-j} \groupp{\frac{\|\eta^{(j)}\|_2^2}{(q-\beta)m}+\zeta \|\eta^{(j)}\|_1^2}.
    \]
    
    \begin{theorem}
    \label{theorem:timevar_theorem}
    (\cite[Theorem 1.4]{timevaryingcorruptionqRK}).
    Suppose $A$ is row-normalized and full-rank, $\x$ is the solution to the over-determined system $Ax = b_t$, and $\|\xi^{(k)}\|_0\leq \beta m$ for all $k\in\N$. Let $\beta < q < 1-\beta$. If $\vp>0$ and $x_k$ is inductively defined by applying one iteration of qRK to the system $Ax=b_t+\eta^{(k)}+\xi^{(k)}$ with initial point $x_{k-1}$, then 
    \[
        \E\|x_k-\x\|^2\leq (1-\vp)^k\|x_0-\x\|^2+\gamma_{k-1}.
    \]
    \end{theorem}
    
    \begin{cor}
    \label{cor:timevar_cor}
    (\cite[Corollary 1.4.1]{timevaryingcorruptionqRK}). In the same setting as \Cref{theorem:timevar_theorem}, when there exists $N>0$ such that $\|\eta^{(j)}\|_\infty\leq N$ for all $j\in [k]$, then 
        \[
            \E\|x_k-\x\|^2\leq (1-\vp)^k\|x_0-\x\|^2 + (1+\zeta m^2)N^2 \frac{1-(1-\vp)^k}{\vp}.
        \]
    \end{cor}
    
    Our setting described by \eqref{eq:genlinsys} has time-invariant noise $\eta$ and corruption $\xi$. Then \Cref{cor:timevar_cor} applies with $\|\eta\|_\infty$ in place of $N$. We are now equipped to present a comparison to this error horizon bound. It is clear from the proof of \Cref{theorem:qRK_EH} that 
    \begin{equation}
        \E\|x_k-\x\|^2\leq (1-C)^k\|x_{0}-\x\|^2 + \frac{1-(1-C)^k}{C}\groupp{\frac{2r(1-q)}{q}+ 1}N^2. 
        \label{eq:NEWboundednoise_EH}
    \end{equation}
    Additionally, the coefficient on $N^2$ given in \Cref{cor:timevar_cor}, can be expanded to 
    \[
        \frac{1-(1-\vp)^k}{\vp}(1+\zeta m^2) =\frac{1-(1-\vp)^k}{\vp}\groupp{1 + \s_{\max}(A)\sqrt{\frac{m}{\beta}}(r+r^2s)+\frac{r^2}{q\beta}},
    \]
    which means the error horizon bound from \Cref{cor:timevar_cor} grows with $m$. The error horizon given in \eqref{eq:NEWboundednoise_EH} remedies this. It is clear for $m\geq 4$ that $2(1-q)\leq \sqrt{m(1-\beta)}$, $1-q-\beta<1$, and $\s_{\max}\geq 1$, whereby 
    \[
        \frac{2r(1-q)}{q}+ 1 \leq 1+\frac{\s_{\max}(A)\beta\sqrt{m(1-\beta)}}{q(1-q-\beta)^2} = 1+\frac{\s_{\max}(A)}{q}\sqrt{\frac{m}{\beta}}r^2s< 1+\zeta m^2,
    \]
    so for sufficiently large $m$, the error horizon bound given in \Cref{cor:timevar_cor} vastly exceeds that of \eqref{eq:NEWboundednoise_EH}. 
    The need for sufficiently large $m$ may be removed in some cases. Because $(1-C)^k, (1-\vp)^k\ra 0$ as $k\ra \infty$, we are only interested in showing $\vp< C$. Indeed, if 
    \begin{equation}
        m < \frac{1}{4\sqrt{\beta}\sqrt{1-q-\beta}}+\frac{\sqrt{1-\beta}}{4\sqrt{1-q-\beta}^3}, \label{eq:conditions_for_timevarying_improvement_1} 
    \end{equation}
    it follows that 
    \[
        0 < \frac{\s_{\max}(A)}{qm}\groupp{\frac{\sqrt{\beta m}}{m(1-q-\beta)}+\frac{\beta\sqrt{m(1-\beta)}}{m(1-q-\beta)^2}}-\frac{\beta}{q}\frac{4\s_{\max}(A)}{\sqrt{m}\sqrt{1-q-\beta}}.
    \]
    Additionally, because $1-q-\beta$ and $1-\beta$ are less than one, it suffices to show 
    \[
        0<\frac{\s_{\max}^2(A)(1-\beta)}{qm}\groupp{\frac{2\sqrt{\beta}}{\sqrt{1-q-\beta}}+\frac{\beta}{1-q-\beta}}-\frac{\beta}{q}\groupp{1+\frac{2\s_{\max}^2(A)}{m(1-q-\beta)}}
    \]
    to assert 
    \[
        0<\frac{\s_{\max}^2(A)}{qm}\groupp{\frac{2\sqrt{\beta(1-\beta)}}{1-q-\beta}+\frac{\beta(1-\beta)}{(1-q-\beta)^2}}-\frac{\beta}{q}\groupp{1+\frac{2\s_{\max}^2(A)}{m(1-q-\beta)}},
    \]
    which is satisfied if 
    \begin{equation}
        \frac{\beta}{1-q-\beta} < 1 \quad \text{ and }\quad \s_{\max}^2(A) > \frac{\beta m}{2}\groupp{\frac{1-q-\beta}{\sqrt{\beta}\sqrt{1-q-\beta}-\beta}}. \label{eq:conditions_for_timevarying_improvement_2}
    \end{equation}
    Then if conditions \eqref{eq:conditions_for_timevarying_improvement_1} and \eqref{eq:conditions_for_timevarying_improvement_2} are met, $\vp < C$, as desired.

    In \cite{sparse_qRK_EH}, the coefficient of $\|\eta\|_\infty^2$ in the error horizon bound for \textit{one} iteration is given by 
    \begin{equation}
        \frac{2r(1-\beta)}{q\sqrt{\beta}}\groupp{1+r\sqrt{\frac{1-\beta}{\beta}}}\sqrt{\frac{n}{m}}\s_{\max}(A)+\frac{1}{2}\groupp{r\frac{(1-\beta)^2}{q(1-q-\beta)}+1}. \label{eq:sparse_EH_bound}
    \end{equation}
    On the other hand, the proof of \Cref{theorem:qRK_EH} shows that the coefficient of $\|\eta\|_\infty^2$ from our error horizon bound for \textit{one} iteration is given by 
    \begin{equation}
        \frac{2r(1-q)}{q}+1. \label{eq:our_EH_bound}
    \end{equation}
    Because $1-q<1-\beta$ and $(1-\beta)/(1-q-\beta) > 1$, it is clear
    \[
        \frac{r(1-q)}{q}+1\leq \frac{r(1-\beta)^2}{q(1-q-\beta)} +1,
    \]
    so 
    \begin{equation}
        \frac{2r(1-q)}{q}+1\leq 4\groupp{\frac{1}{2}\groupp{r\frac{(1-\beta)^2}{q(1-q-\beta)}+1}}. \label{eq:err_bound_const_mult}
    \end{equation}
    This means the coefficient \eqref{eq:our_EH_bound} is, at worst, a universal constant multiple of a \textit{part} of the coefficient \eqref{eq:sparse_EH_bound}. More importantly, our bound has the added benefit of holding independently of $m$, $n$, and $\sigma_{\max}(A)$.

\subsection{Analogous Results for dqRK}
\label{sec:dqRK_analogous_results}

    Because the work providing convergence guarantees for dqRK closely parallels that of qRK, it is perhaps unsurprising that there are natural dqRK counterparts to the results for qRK. Note that if $p$ is to represent the worst-case measure of what proportion of the admissible rows are guaranteed to be non-corrupt, then for dqRK we should set $p=\frac{q-q_0-\beta}{q-q_0}$. Similarly, if $r$ is to bound the ratio of false positives to all erroneous classifications, we may use re-use $r=\frac{\beta}{1-q-\beta}>\frac{\beta}{1-q+q_0-\beta}$. The first extension of qRK to dqRK comes in the form of a respective improved convergence bound:
    \begin{cor}
        \label{cor:dqRK_STD_NEWbound}
        (Alternative dqRK Bound). Suppose $\beta < q_0 < q < 1-\beta$ and $q-q_0>\beta$. Assume $A$ is row-normalized and full-rank. Let $\x$ be the solution to the over-determined system  $Ax=b_t$, $\|\xi\|_0\leq \beta m$, and $\eta=0$. If 
        \[
            \frac{q}{q-q_0-\beta}\groupp{\frac{\beta m}{\s_{\max}^2(A)}+2r}< \kappa_q^{-2}(A)+ \frac{\kappa_{q_0}^{-2}(A)}{q_0m},
        \]
        and $\lll x_0, a_j\rrr = b_j$ for some $j\in [m]$, then the $k-th$ iterate of dqRK applied to the system $Ax=b$ witnesses
        \[
            \E\|x_k-\x\|^2 \leq (1-C)^k\|x_0-\x\|^2,
        \]
        where 
        \[
            C=p\groupp{\hat\kappa_q^{-2}(A)+ \frac{\hat\kappa_{q_0}^{-2}(A)}{qm}}-\frac{1}{q-q_0}\groupp{\beta+\frac{2\s_{\max}^2(A)r}{m}}
        \]
    \end{cor}

    Interestingly, the bound for the rate of convergence of dqRK from \Cref{cor:dqRK_STD_NEWbound} is provably smaller than that of \Cref{theorem:dqrk} under the same mild constraints from \Cref{theorem:boundcomparisontheorem}! Additionally, the discussion in \Cref{sec:qRK_new_bounds} of how restrictive these hypotheses are also applies in the dqRK setting.

    \begin{cor}
        \label{cor:dqRKboundcomparisontheorem}
        Under the same conditions as \Cref{theorem:boundcomparisontheorem}, $\alpha_1^{\text{dqRK}}<\alpha_2^{\text{dqRK}}$ where 
        \begin{align*}
            \alpha_1^{\text{dqRK}} &\Def 1 - p\groupp{\hat\kappa_q^{-2}(A)+ \frac{\hat\kappa_{q_0}^{-2}(A)}{qm}} + \frac{1}{q-q_0}\groupp{\beta+\frac{2\s_{\max}^2(A)r}{m}}\\
            \alpha_2^{\text{dqRK}} &\Def 1 - p\groupp{\hat\kappa_q^{-2}(A)+ \frac{\hat\kappa_{q_0}^{-2}(A)}{qm}} + \frac{\s_{\max}^2(A)}{(q-q_0)m}\groupp{2\sqrt{r}+r},
        \end{align*}
        which are the decay factors from \Cref{cor:dqRK_STD_NEWbound} and \Cref{theorem:dqrk}, respectively.
    \end{cor}
    \begin{proof}
        Similarly to the proof of \Cref{theorem:boundcomparisontheorem}, the second hypothesis and \eqref{eq:firstcondition_rewritten} makes it clear 
        \[  
            \frac{\s_{\max}^2(A)}{(q-q_0)m}\groupp{\frac{2\sqrt{\beta}}{\sqrt{1-q-\beta}}+\frac{\beta}{1-q-\beta}-\frac{2\beta}{1-q-\beta}}>\frac{\beta}{q-q_0},
        \]
        so 
        \[
            \frac{\beta}{q-q_0}+\frac{ 2\beta\s_{\max}^2(A)}{m(q-q_0)(1-q-\beta)}-\frac{\s_{\max}^2(A)}{(q-q_0)m}\groupp{\frac{2\sqrt{\beta}}{\sqrt{1-q-\beta}}+\frac{\beta}{1-q-\beta}}<0.
        \]
        Then it is clear $\alpha_1^{\text{dqRK}}<\alpha_2^{\text{dqRK}}$.
    \end{proof}

     The final result from qRK that extends naturally to dqRK is \Cref{theorem:qRK_EH}. Though only the result that is most related to \Cref{theorem:qRK_EH} is presented, the same discussion in \Cref{sec:qRK_EH}, i.e. the lack of importance of the hypothesis $\supp(\eta)\cap \supp(\xi)=\varnothing$ and the subsequent formulation of \Cref{cor:generalqRK_EH}, extends to the dqRK setting.

    \begin{theorem}
        \label{theorem:dqRK_EH}
        (dqRK Error Horizon). Suppose $A$ is row-normalized and full-rank. Let $\x$ be the solution to the over-determined system $Ax=b_t$ and $b=b_t+\eta+\xi$ with $\|\xi\|_0\leq \beta m$ and $\supp(\eta)\cap \supp(\xi)=\varnothing$. Suppose further that $\beta<q_0<q<1-\beta$, $q-q_0>\beta$, and 
        \[
            \frac{q}{q-q_0-\beta}\groupp{\frac{\beta m}{\s_{\max}^2(A)}+2r+\frac{4\sqrt{\beta m}\sqrt{r}}{\s_{\max}(A)}}<\kappa_q^{-2}(A).
        \]
        The expected error of the dqRK algorithm applied to the system $Ax=b_t+\eta+\xi$ decays as
        \[
            \E\|x_{k+1}-\x\|^2\leq (1-C)\E\|x_k-\x\|^2+\frac{1}{C}\groupp{2r\frac{1-q}{(q-q_0)}+1}\|\eta\|_\infty^2,
        \]
        where 
        \[
             C=p\hat\kappa_q^{-2}(A)-\frac{1}{q-q_0}\groupp{\beta+\frac{2\s_{\max}^2(A)r}{m}+\frac{4\s_{\max}(A)\sqrt{\beta r}}{\sqrt{m}}}.
        \]
    \end{theorem}

\section{Empirical Results}
    \label{sec:empirical_results}
    This section presents the empirical behavior of RK, qRK, and dqRK applied to \eqref{eq:genlinsys} using reproducible code that may be found at \url{https://github.com/Mr-E-User/Quantile-RK-and-Double-Quantile-RK-Error-Horizon-Analysis}. Our first two experiments show the convergence of RK, qRK, and dqRK on systems with only dense noise and systems with both noise and large, sparse corruptions. Together, these experiments demonstrate that the horizons for qRK and dqRK outperform RK more significantly in the presence of large, sparse corruptions. In our last experiment, we observed the behavior of the empirical error horizons of RK and dqRK as a function of the scaling of sparse corruptions, demonstrating that scaling has a minimal impact on the horizon compared to RK. For all the numerical experiments, the following parameters were used: $m=5000$, $n=2500$, $\beta=0.05$, $q_0=0.6$, and $q=0.8$. The approximation errors in \Cref{fig:NOISE_ONLY_error_horizon_comparison} and \Cref{fig:error_horizon_comparison} are given by a single sample run of each algorithm, but the results across runs are consistent with those presented. Since the focus of this work is on the study of the error horizon, we refer the reader to the papers \cite{dqRK, timevaryingcorruptionqRK,qRK2} for comparisons of these methods under different quantiles, $\beta$, and problem size choices.

    In \Cref{fig:NOISE_ONLY_error_horizon_comparison}, we present the performance of RK, qRK, and dqRK on random $m\times n$ Gaussian and uniformly random matrix $A$. The solution $\x\in \R^n$ is randomly selected, and we define $b_t\Def A\x$. To introduce noise to the system, $b$ was determined by adding a Gaussian vector $\eta \sim \mathcal{N}(0,I_m)$ to $b_t$. Here, we do not consider sparse noise so we set $\xi = 0$. Then, RK, qRK, and dqRK are used to solve the system $Ax=b$, and their respective approximation errors were plotted. The goal of \Cref{fig:NOISE_ONLY_error_horizon_comparison} is to see how much of an impact failing condition \eqref{eq:EH_condition} has on the resulting error horizons for the comparative sizes of RK and qRK/dqRK. In this case, we see \eqref{eq:EH_condition} fails to hold, and despite this, the error horizon of qRK, dqRK, and RK tend to be on the same order. This is expected when the largest $(1-q) m+1$ entries of the corruption are of the same order. The methods qRK and dqRK, unfortunately, select an index corresponding to these entries infinitely often. Then, just as is the case with RK, each algorithm's iterates will have distance from $\x$ on the order of $\|b-b_t\|_\infty$.

    \begin{figure}[!htb]
        \centering
        \input{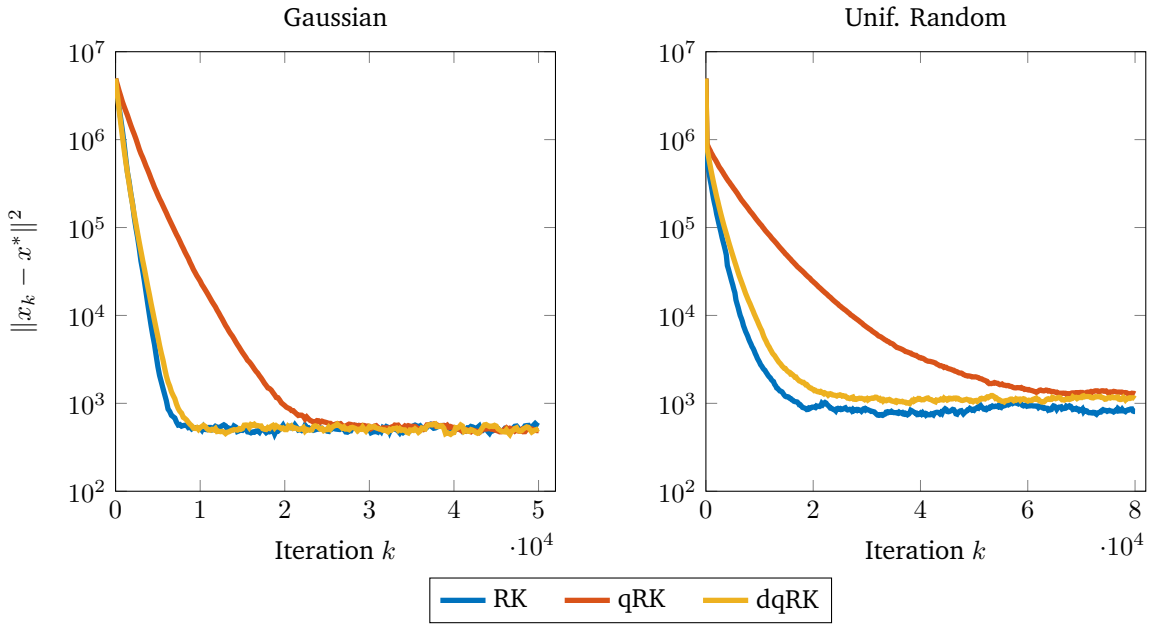}
        \caption{Comparison of approximation error for a given run of qRK, dqRK, and RK on a system with Gaussian $A$ and uniformly random $A$, both with \textit{only} fixed noise.}
        \label{fig:NOISE_ONLY_error_horizon_comparison}
    \end{figure}

    \Cref{fig:error_horizon_comparison} shows the performance of RK, qRK, and dqRK when the sparse corruption's largest entry is several orders of magnitude larger than that of the noise. The same procedure as in \Cref{fig:NOISE_ONLY_error_horizon_comparison} is used to generate $A$, $\x$, and $b_t$. To generate the sparse noise vector $\xi$, we choose $\beta m$ entries to have values uniformly selected between $0$ and $100$. The final vector $b$ is formed by adding a Gaussian vector $\eta \sim \mathcal{N}(0,I_m)$ to $b_t+\xi$. Condition \eqref{eq:EH_condition} is much more likely to be satisfied when large sparse corruption is also present, so the use of the quantile information allows qRK and dqRK to bypass the $(1-q)m$ largest corruptions, yielding a smaller error horizon. As expected, the larger updates (on average) from dqRK yields faster convergence than qRK towards the error horizon. 
    
    \begin{figure}[!htb]
        \centering
        \input{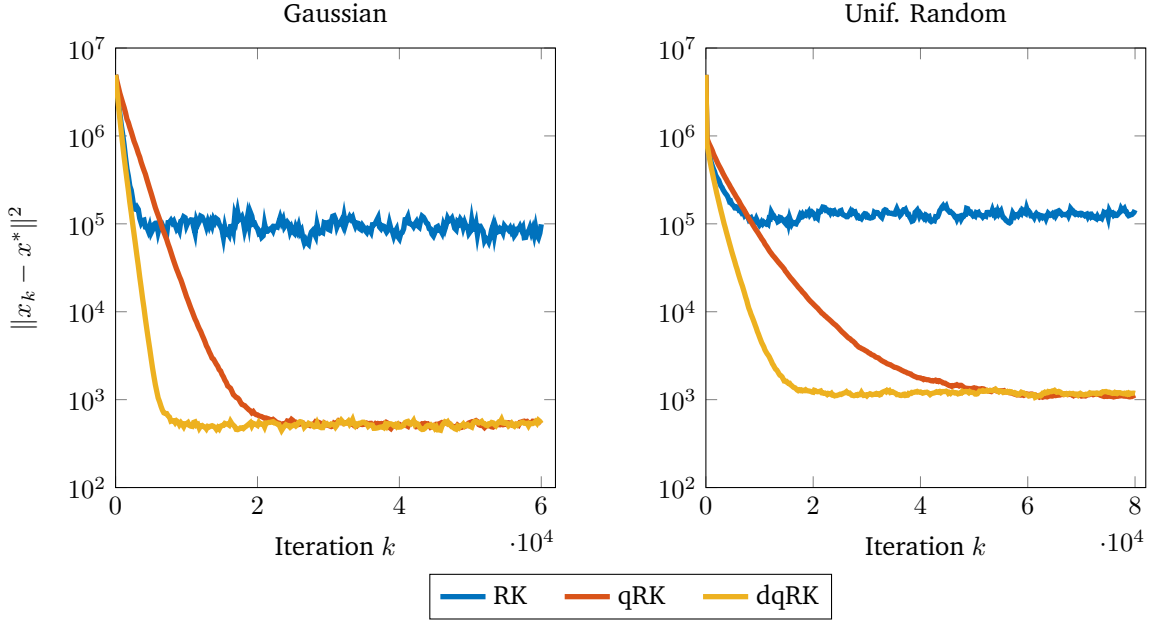}
        \caption{Comparison of approximation error for a given run of qRK, dqRK, and RK on a system with Gaussian $A$ and uniformly random $A$, both with fixed noise and large, sparse corruption.}
        \label{fig:error_horizon_comparison}
    \end{figure}

    \Cref{fig:error_horizon_vs_sparse_scale} illustrates the relationship between the magnitude of the corruption and average error horizon. To generate this plot, for each given corruption scale, we start by generating a Gaussian matrix $A$, solution vector $\x$, and $b_t := A\x$, as in \Cref{fig:NOISE_ONLY_error_horizon_comparison}. Then, $\xi$ is constructed by randomly selecting $\beta m$ entries of $b_t$ to add a value selected uniformly randomly from $[0,1]$ and scaled by the corruption scale. Finally, a random Gaussian vector is added to the resultant $b_t+\xi$ to construct $b$. For each algorithm, the point 
    \[
        \groupp{\frac{\ve_{(1)}}{\ve_{((1-q)m+1)}},\text{EH}}
    \]
    is plotted, where $\text{EH}$ is the empirical error horizon given by the maximum squared approximation error of the last 100 iterates,  $\ve=\eta+\xi$, and $\ve_{(j)}$ is as defined in \Cref{cor:generalqRK_EH}. For both algorithms, 15 sample runs were performed at each corruption scale. As was demonstrated by \Cref{fig:NOISE_ONLY_error_horizon_comparison}, when there is no sparse corruption, or the $(1-q)m +1$ largest entries of the full corruption are roughly the same size, the error horizons of RK and dqRK are on the same order. When the scale of the sparse corruption increases, dqRK starts to distinguish between the larger, sparse corruption and the noise, and is less likely to select indices in the support of the sparse corruption. Importantly, we see that empirically, a large gap between the largest noise entry and the smallest sparse corruption entry is not required to achieve a noticeably smaller error horizon using dqRK. 
    
    \begin{figure}[!htb]
    \centering
    \input{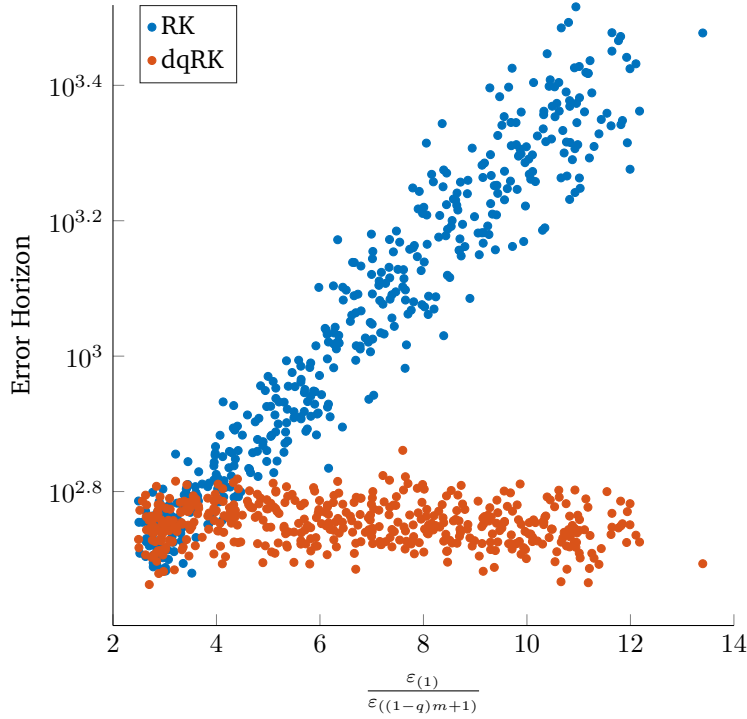}
    \caption{Correlation between the ratio $\ve_{(1)}/\ve_{((1-q)m+1)}$ and the error horizon for dqRK and RK on a system with Gaussian $A$.}
    \label{fig:error_horizon_vs_sparse_scale}
    \end{figure}

\section{Proofs}
    \label{sec:proofs}
    For the remainder of the paper, it is assumed $A$ is row-normalized. Additionally, $\E_{i\in X}Z$ will be used to denote $\frac{1}{|X|}\sum_{i\in X}Z(i)$. 
    
    \subsection{Proof of the Alternate qRK Bound (\Cref{cor:qRK_STD_NEWbound})}
    
        For this section, we consider the system \eqref{eq:genlinsys} with $\eta=0$. Additionally, $x_k$ is a fixed vector in $\R^n$ and $Q$ denotes $q\quant(\{|r_j|\}_{j=1}^m)$ where $r_j=b_j-\lll x_k,a_j\rrr$. 
    
        \label{sec:qRK_NEWBound_proof}
        The proof of \Cref{theorem:qrk} hinges on bounding the expected squared approximation error in the event a corrupt row is chosen; this was achieved in part using the Cauchy-Schwarz inequality \cite{qRK}. By using an alternate bound in \Cref{lem:qRK_STD_NEWcorruptE}, we provide a tighter convergence rate under certain conditions.  Thus, the same argument as in \cite[Main Result]{qRK} (\Cref{theorem:qrk} in this paper) using \Cref{lem:qRK_STD_NEWcorruptE} in place of \cite[Lemma 2]{qRK} is enough. For clarity, \cite[Lemma 1]{qRK} and \cite[Lemma 3]{qRK} are provided as \Cref{lem:qRK_STD_quantilebound} and \Cref{lem:qRK_STD_noncorruptE}, respectively.         

        In \Cref{lem:qRK_STD_quantilebound}, the quantile $Q$ serves to control the error incurred by selecting a corrupt index in an iteration of qRK. This will later be useful in bounding the squared expected squared approximation error in the event that a corrupt index is chosen.
        \begin{lem}
            \label{lem:qRK_STD_quantilebound}
            \cite[Lemma 1]{qRK} Let $0<q<1-\beta$. Suppose $\x$ is a solution to $A\x=b_t$. If $\|\xi\|_0\leq \beta m$, then 
            \[
                Q \leq \frac{\s_{\max}(A)}{\sqrt{m}\sqrt{1-q-\beta}}\|x_k-\x\|.
            \]
        \end{lem}

        To precise which indices are corrupt yet admissible, let $S = \{j\in C: |r_j|\leq Q\}$, where $C$ is the set indices where corruption is present, i.e. $C=\{j\in[m]:\xi_j\neq 0\}$. The difference in our convergence rate and that from \cite{qRK} is due to the difference in approaches to bounding $\E_{i\in S}\|x_{k+1}-\x\|^2$. Namely, the Cauchy-Schwarz method used in \cite{qRK} is more lossy than applying the inequality 
        \begin{equation}
            (x+y)^2\leq 2(x^2+y^2), \label{eq:usefulinequality}
        \end{equation}
        in the setting described by \Cref{theorem:boundcomparisontheorem}.
        \begin{lem}
            \label{lem:qRK_STD_NEWcorruptE}
            Assuming the same hypotheses as in \Cref{lem:qRK_STD_quantilebound}, letting $x_{k+1}$ be the random vector generated by qRK on $Ax=b$ with initial point $x_k$ and $b=b_t+\xi$ yields the following: 
            \[
                \E_{i\in S}\|x_{k+1}-\x\|^2\leq 2\groupp{1+\frac{\s_{\max}^2(A)}{m(1-q-\beta)}}\|x_k-\x\|^2.
            \]
        \end{lem}
        \begin{proof}
            This is a simple application of the triangle inequality and \eqref{eq:usefulinequality}. It is clear
            \begin{align*}
                \E_{i\in S}\|x_{k+1}-\x\|^2 &= \frac{1}{|S|}\sum_{i\in S}\|x_k+(b_i-\lll x_k,a_i\rrr)a_i-\x\|^2\\
                &\leq \frac{2}{|S|}\sum_{i\in S} \groupp{\|x_k-\x\|^2+|b_i-\lll x_k,a_i\rrr|^2}\\
                &\leq 2\groupp{\|x_k-\x\|^2+Q^2}\\
                &\leq 2\groupp{1+\frac{\s_{\max}^2(A)}{m(1-q-\beta)}}\|x_k-\x\|^2.
            \end{align*}
        \end{proof}

        Just as a bound was provided for $\E_{i\in S}\|x_{k+1}-\x\|^2$, \Cref{lem:qRK_STD_noncorruptE} bounds the expected squared approximation error in the event a non-corrupt index is chosen. 
        \begin{lem}
            \label{lem:qRK_STD_noncorruptE}
            \cite[Lemma 3]{qRK} Assuming the same hypotheses as \Cref{lem:qRK_STD_NEWcorruptE}, letting $B=\{j\in [m]:|r_j|\leq Q\}$ yields
            \[
                \E_{i\in B\sm S} \|x_{k+1}-\x\|^2 \leq \groupp{1-\frac{\s_{q-\beta,\min}^2(A)}{qm}}\|x_k-\x\|^2.
            \]
        \end{lem}

        Finally, \Cref{cor:qRK_STD_NEWbound} can be proved using \Cref{lem:qRK_STD_NEWcorruptE} and \Cref{lem:qRK_STD_noncorruptE} and we are now equipped to prove \Cref{cor:qRK_STD_NEWbound}:
        \begin{proof}[Proof of \Cref{cor:qRK_STD_NEWbound}]
            Observe,
            \begin{align*}
                \E\|x_{k+1}-\x\|^2 &= \PP(i\in S)\E_{i\in S}\|x_{k+1}-\x\|^2+\PP(i\in B\sm S)\E_{i\in B\sm S}\|x_{k+1}-\x\|^2\\
                &\leq \groupb{\frac{|S|}{qm}\groupp{2\groupp{1+\frac{\s_{\max}^2(A)}{m(1-q-\beta)}}}+\groupp{1-\frac{|S|}{qm}}\groupp{1-\frac{\s_{q-\beta,\min}^2(A)}{qm}}}\|x_k-\x\|^2,
            \end{align*}
            which is clearly increasing in $|S|$. Because $|S|\leq \beta m$, it follows
            \[
                \E\|x_{k+1}-\x\|^2 \leq \groupb{1+\frac{\beta}{q}\groupp{1+\frac{2\s_{\max}^2(A)}{m(1-q-\beta)}}-(q-\beta)\frac{\s_{q-\beta,\min}^2(A)}{q^2m}}\|x_k-\x\|^2.
            \]
            To ensure decay, we require $q>\beta$ and 
            \[
                \frac{q}{q-\beta}\groupp{\frac{\beta m}{\s_{\max}^2(A)}+\frac{2\beta}{1-q-\beta}}< \frac{\s_{q-\beta,\min}^2(A)}{\s_{\max}^2(A)}.
            \]
        \end{proof}
        
    \subsection{Proof of the Alternate dqRK Bound (\Cref{cor:dqRK_STD_NEWbound})}
        \label{sec:dqRK_NEWBound_proof}
        Just as in \Cref{sec:qRK_NEWBound_proof}, this section considers the system \eqref{eq:genlinsys} with $\eta=0$. Then $r_j=\lll x_k, a_j\rrr-b_j$, where $x_k\in \R^n$ is fixed, and the two quantiles $Q_0$ and $Q$ are given by 
        \begin{align*}
            Q_0&=q_0\quant(\{|r_j|\}_{j=1}^m) \\
            Q&=q\quant(\{|r_j|\}_{j=1}^m).
        \end{align*}
        We use $C$ to denote the indices corresponding to the corrupted entries, i.e. $C=\supp(\xi)$, and $S\sset C$ is defined to be the corrupt admissible indices given by $\{j\in C: Q_0<|r_j|\leq Q\}$. Letting $B=\{j\in [m]: Q_0<|r_j|\leq Q\}$ be the collection of admissible indices, we may express the \textit{non-corrupt} admissible indices by $B\sm S$. 
        
        \begin{cor}
            \label{cor:dqRK_STD_quantilebound}
            \cite[Corollary 2]{dqRK} Let $0<q_0<q<1-\beta$. If $\x$ is a solution to the system $Ax=b_t$ and $\|\xi\|_0\leq \beta m$, then 
            \[
                Q\leq \frac{\s_{\max}(A)}{\sqrt{m}\sqrt{1-q-\beta}}\|x_k-\x\|.
            \]
        \end{cor}
    
        \begin{cor}
            \label{cor:dqRK_STD_NEWcorruptE}
            Assuming the same hypotheses as in \Cref{cor:dqRK_STD_quantilebound}, letting $x_{k+1}$ be random vector generated by dqRK on $Ax=b$ with initial point $x_k$ and $b=b_t+\xi$ yields the following: 
            \[
                \E_{i\in S}\|x_{k+1}-\x\|^2\leq 2\groupp{1+\frac{\s_{\max}^2(A)}{m(1-q-\beta)}}\|x_k-\x\|^2.
            \]
        \end{cor}
    
        \begin{lem}
            \label{lem:dqRK_STD_noncorruptE}
            \cite[Lemma 2]{dqRK} In the same setting as \Cref{cor:dqRK_STD_NEWcorruptE}, if $\lll x_k,a_j\rrr =b_j$ for some $j\in[m]$, then 
            \[
                \E_{i\in B\sm S}\|x_{k+1}-\x\|^2\leq \groupp{1-\frac{\s_{q-\beta,\min}^2(A)}{q m}-\frac{\s_{q_0-\beta,\min}^2(A)}{q_0q m^2}}\|x_k-\x\|^2.
            \]
        \end{lem}
    
        We are now equipped to prove \Cref{cor:dqRK_STD_NEWbound}:
        \begin{proof}[Proof of \Cref{cor:dqRK_STD_NEWbound}]
            Similarly to the proof of \cite[Theorem 2]{dqRK}, \Cref{cor:dqRK_STD_NEWcorruptE} with \Cref{lem:dqRK_STD_noncorruptE} yields 
            \begin{align*}
                \E\|x_{k+1}-\x\|^2 &= \frac{|S|}{(q-q_0)m}\E_{i\in S}\|x_{k+1}-\x\|^2+\groupp{1-\frac{|S|}{(q-q_0)m}}\E_{i\in B\sm S}\|x_{k+1}-\x\|^2 \\
                &\leq \Bigg[\groupp{1-\frac{|S|}{(q-q_0)m}}\groupp{1-\frac{\s_{q-\beta,\min}^2(A)}{q m}-\frac{\s_{q_0-\beta,\min}^2(A)}{q_0q m^2}} \\
                &+\frac{2|S|}{(q-q_0)m}\groupp{1+\frac{\s_{\max}^2(A)}{m(1-q-\beta)}}\Bigg]\|x_k-\x\|^2,
            \end{align*}
            Which is increasing in $|S|$. Thus, 
            \begin{align*}
                \E\|e_{k+1}\|^2 &\leq \groupb{1+\frac{\beta}{q-q_0}\groupp{1+\frac{2\s_{\max}^2(A)}{m(1-q-\beta)}}-(q-q_0-\beta)\groupp{\frac{\s_{q-\beta,\min}^2(A)}{(q-q_0)q m}+\frac{\s_{q_0-\beta,\min}^2(A)}{(q-q_0)q_0q m^2}}}\|e_k\|^2,
            \end{align*}
            where $e_k=x_k-\x$. To ensure decay, we require
            \[
                \frac{q}{q-q_0-\beta}\groupp{\frac{\beta m}{\s_{\max}^2(A)}+\frac{2\beta}{1-q-\beta}}< \frac{1}{\s_{\max}^2(A)}\groupp{\s_{q-\beta,\min}^2(A)+\frac{\s_{q_0-\beta,\min}^2(A)}{q_0m}}.
            \]
        \end{proof}
    
    \subsection{Proof of the qRK Error Horizon Bound (\Cref{theorem:qRK_EH})}
        \label{sec:qRK_EH_proof}
        We now consider the full system \eqref{eq:genlinsys} where $\eta$ may be nonzero. We will denote the solution hyperplanes to the sparsely-corrupted system \eqref{eq:truesyseq} by $~{H_i\dl=\{x:\lll x,a_i\rrr =(b_t+\xi)_i\}}$. Similarly, the solution hyperplanes to the corrupted system \eqref{eq:genlinsys} will be given by $H_i=\{x:\lll x, a_i\rrr = b_i\}$. 

        Ideally, we wish to say that qRK applied to $Ax=b$ will yield convergence to the true solution $\x$ satisfying $A\x=b_t$ up to an error horizon dependent only on the small noise $\eta$. This would closely mirror the main result of $\cite{needell_LS}$, and it may be tempting to provide a proof almost identical to \cite[Theorem 2.1]{needell_LS}, with the appropriate changes being made. Such a proof might read as follows:
        \begin{fakeproof}
            Note that in the proof of \Cref{theorem:qrk}, it is shown that with the same hypothesis as in \Cref{theorem:qrk}, for any $x_{k-1}\in \R^n$, running one iteration of qRK on the system $Ax=b_t+\xi$ to produce $x_k\dl$ witnesses 
            \begin{equation}
                \E\|x_k\dl-\x\|^2\leq (1-C)\|x_{k-1}-\x\|^2, \label{eq:qrk_bound}
            \end{equation}
            where the expectation is conditional on the first $k-1$ iterates.
            
            Now note that if $H_i$ is the solution space used to produce $x_k$, then by letting $x_k\dl$ be the projection of $x_{k-1}$ onto $H_i\dl$, the same argument found in the proof of \cite[Theorem 2.1]{needell_LS} with the assumption $\supp(\eta)\cap \supp(\xi)=\varnothing$ shows $x_k-\x= x_k\dl -\x + \eta_i a_i$ and $\|x_k-\x\|^2 = \|x_k\dl -\x\|^2 + |\eta_i|^2$. Then it is clear 
            \begin{align}
                \E\|x_k-\x\|^2 &\leq \E\|x_k\dl -\x\|^2 + \|\eta\|_\infty^2 \label{eq:wrongexpectation}\\
                &\leq (1-C)\|x_{k-1}-\x\|^2 + \|\eta\|_\infty^2, \label{eq:badinequality}
            \end{align}
            where \eqref{eq:badinequality} follows from \eqref{eq:qrk_bound}. Taking the total expectation and recursively applying the results yields
            \begin{align*}
                \E\|x_k-\x\|^2 &\leq (1-C)^k\|x_0-\x\|^2 + \sum_{j=0}^{k-1} (1-C)^j \|\eta\|_\infty^2 \\
                &\leq (1-C)^k\|x_0-\x\|^2 +\frac{1}{C} \|\eta\|_\infty^2.
            \end{align*}
        \end{fakeproof}
        
        Unfortunately, it is erroneous to argue \eqref{eq:badinequality} follows from \eqref{eq:qrk_bound}. Namely, the expectation on the left-hand side of \eqref{eq:qrk_bound} is taken over the set of indices admissible based on 
        \[
            Q' = q\quant\groupp{\groupc{|\lll x_k, a_j\rrr -(b_t+\xi)_j|}_{j=1}^m},
        \]
        whereas the expectation on the right-hand side of \eqref{eq:wrongexpectation} is taken over the set of indices admissible based on 
        \[
            Q = q\quant\groupp{\groupc{|\lll x_k, a_j\rrr -b_j|}_{j=1}^m}.
        \]
        It is not clear that these two quantiles $Q'$ and $Q$ are the same, so it would not be reasonable to assume 
        \[
            \groupc{j\in [m]:|\lll x_k, a_j\rrr -b_j|\leq Q } = \groupc{j\in [m]:|\lll x_k, a_j\rrr -(b_t+\xi)_j|\leq Q' },
        \]
        whereby the two expectations are not necessarily equal. Fortunately, a valid argument towards a similar bound can be achieved.

        For this section, let $r_j=\lll x_k,a_j\rrr - (b_t-\xi)_j$, $Q=q\quant(\{|r_j-\eta_j|\}_{j=1}^m)$, and $Q'=q\quant(\{|r_j|\}_{j=1}^m)$, i.e. $r_j$ is the $j$-th residual entry of $Ax_k-b_t-\xi$ and $Q$ and $Q'$ are the quantiles determined by qRK on the systems 
        \begin{equation*}
            Ax=b_t+\xi+\eta, 
            \label{eq:fullcorrupsys}
        \end{equation*} 
        and 
        \begin{equation}
            Ax=b_t+\xi, \label{eq:sparsecorruptsys}
        \end{equation}
        respectively. Because the source of error of the false argument of \eqref{eq:qrk_bound} stems from the admissible indices differing when applying a step of qRK on \eqref{eq:fullcorrupsys} versus \eqref{eq:sparsecorruptsys}, we wish to provide analysis on a modified qRK applied to the system \eqref{eq:sparsecorruptsys}. That is, we will bound $\E\|x_{k+1}-\x\|^2$ when 
        \begin{equation}
            x_{k+1}=x_k+((b_t+\xi)_i-\lll x_k,a_i\rrr)a_i\quad \text{for}\quad i\sim \unif(\{j\in [m]:|r_j-\eta_j|\leq Q\}). \label{eq:qRK_EH_xkplus1scheme}
        \end{equation}
        For the remainder of this section, $S=\{j\in C: |r_j-\eta_j|\leq Q\}$ and $B= \{j\in [m]: |r_j-\eta_j|\leq Q\}$ where $C$ is the indices of non-zero entries of $\xi$, i.e. $C=\supp(\xi)$.
        
        This section is structured as follows: \Cref{cor:qRK_EH_quantilebound} provides a revised quantile bound, which is implemented in the proof of \Cref{cor:qRK_EH_corruptE} to control the size of $\E_{i\in S}\|x_{k+1}-\x\|^2$. After such, the results of the modified qRK on \eqref{eq:sparsecorruptsys} are presented in \Cref{lem:qRK_EH_differentindexE}. This result is then used in \Cref{theorem:qRK_EH}, which states qRK may be used to converge to a smaller error horizon given the appropriate conditions are met.
        \begin{cor}
            \label{cor:qRK_EH_quantilebound}
            Let $0<q<1-\beta$ and suppose $\x$ is a solution to $A\x=b_t$. If $\|\xi\|_0\leq \beta m$, then 
            \[
                Q\leq \frac{\s_{\max}\|x_k-\x\|}{\sqrt{m}\sqrt{1-q-\beta}}+ \frac{\sqrt{1-q}\|\eta\|_\infty}{\sqrt{1-q-\beta}}.
            \]
        \end{cor}
        \begin{proof}
            Because there are at least $m(1-q-\beta)$ indices $\ell$ such that $\ell\in N\Def \groupc{j\in [m]\sm C: |r_j-\eta_j|> Q}$, it is clear 
            \begin{align*}
                m(1-q-\beta)Q^2 &\leq \sum_{j \in N} \groupp{r_j-\eta_j}^2 \\
                &=\sum_{j\in N} \groupp{\lll x_k,a_j\rrr - (b_t)_j -\eta_j }^2\\
                &=\sum_{j\in N} \groupp{\lll x_k-\x, a_j\rrr-\eta_j}^2 \\
                &\leq \sum_{j\in N} \groupp{\groupp{\lll x_k-\x, a_j\rrr}^2 + 2\|\eta\|_\infty|\lll x_k-\x,a_j\rrr|+\|\eta\|_\infty^2}\\
                &= \|A_{N}(x_k-\x)\|^2+2\|\eta\|_\infty \|A_{N}(x_k-\x)\|_1+|N|\|\eta\|_\infty^2\\
                &\leq \s_{\max}^2\|x_k-\x\|^2 + 2\|\eta\|_\infty \sqrt{|N|}\s_{\max}\|x_k-\x\|+|N|\|\eta\|_\infty^2\\
                &= (\s_{\max}\|x_k-\x\|+\sqrt{|N|}\|\eta\|_\infty)^2\\
                &\leq (\s_{\max}\|x_k-\x\|+\sqrt{m(1-q)}\|\eta\|_\infty)^2.
            \end{align*}
        \end{proof}

        Though we could attempt to shorten this proof by writing 
        \begin{align*}
            m(1-q-\beta)Q^2 &\leq \sum_{j\in N} \groupp{\lll x_k-\x,a_j\rrr - \eta_j}^2 \\
            &=\|A_{N}(x_k-\x)-\eta_N\|^2 \\
            &\leq (\s_{\max}(A)\|x_k-\x\|+m(1-q)\|\eta\|_\infty)^2,
        \end{align*}
        this comes at the added cost of contributing a factor of $m(1-q)\geq 1$ to the $\|\eta\|_\infty$ term. In any case, either bound would serve to control the size of $\E_{i\in S}\|x_{k+1}-\x\|^2$. \Cref{cor:qRK_EH_corruptE} follows from \Cref{cor:qRK_EH_quantilebound} and asserts $\E_{i\in S}\|x_{k+1}-\x\|^2$, up to a constant function of the noisy corruption $\eta$, does not grow too much larger than $\|x_k-\x\|^2$.

        \begin{cor}
            \label{cor:qRK_EH_corruptE}
            In the same setting as \Cref{cor:qRK_EH_quantilebound}, if $x_{k+1}$ is the random variable defined by \eqref{eq:qRK_EH_xkplus1scheme} and $\supp(\xi)\cap \supp(\eta)=\varnothing$, then 
            \[
                \E_{i\in S} \| x_{k+1}-\x\|^2 \leq 2\groupp{1+\frac{\s_{\max}^2(A)}{m(1-q-\beta)}+\frac{2\s_{\max}(A)}{\sqrt{m}\sqrt{1-q-\beta}}}\|x_k-\x\|^2+\frac{2(1-q)\|\eta\|_\infty^2}{1-q-\beta}
            \]
        \end{cor}
        \begin{proof}
            Note that $\supp(\xi)\cap \supp(\eta)=\varnothing$ implies the entries of $b=b_t+\eta+\xi$ have the form 
            \[
                b_j=\begin{cases}(b_t+\eta)_j, & \text{if } j\in [m]\sm C \\
                (b_t+\xi)_j, & \text{if } j\in C.
                \end{cases}
            \]
            Thus, whenever $j\in S$, we have $|b_j-\lll x_k,a_j\rrr|=|r_j|\leq Q$. From here, it is straight-forward to see
            \begin{align*}
                \E_{i\in S} \| x_{k+1}-\x\|^2 &= \frac{1}{|S|}\sum_{i\in S} \|x_k+(b_i-\lll x_k,a_i\rrr)a_i -\x\|^2 \\
                &\leq \frac{1}{|S|} \sum_{i\in S}\groupp{|b_i-\lll x_k, a_i\rrr| + \|x_k-\x\|}^2 \\
                &\leq \frac{1}{|S|} \sum_{i\in S}\groupp{Q + \|x_k-\x\|}^2 \\
                &\leq \frac{1}{|S|} \sum_{i\in S}\groupp{\frac{\s_{\max}\|x_k-\x\|}{\sqrt{m}\sqrt{1-q-\beta}}+ \frac{\sqrt{1-q}\|\eta\|_\infty}{\sqrt{1-q-\beta}} + \|x_k-\x\|}^2 \\
                &\leq 2\groupp{\frac{\s_{\max}(A)}{\sqrt{m}\sqrt{1-q-\beta}}+1}^2\|x_k-\x\|^2+\frac{2(1-q)\|\eta\|_\infty^2}{1-q-\beta}.
            \end{align*}
        \end{proof} 

        The importance of the hypothesis $\supp(\xi)\cap \supp(\eta)=\varnothing$ lies in the ability to infer $|(b_t+\xi)_j-\lll x_k,a_j\rrr|\leq Q$ for $j\in S$. The quantile $Q$ is determined from residuals $(b_t+\eta+\xi)_j-\lll x_k,a_j\rrr$, and if $\supp(\xi)\cap \supp(\eta)\neq \varnothing$, it may not be the case that $|(b_t+\xi)_j-\lll x_k,a_j\rrr|\leq Q$ for all $j\in S$. It is shown in \Cref{sec:qRK_new_bounds} that this hypothesis may be removed.  

        Because the sets $B$ and $S$ from this section have different meanings than in \Cref{sec:qRK_NEWBound_proof}, we note here that \Cref{lem:qRK_STD_noncorruptE} still holds with the new definitions of $B$ and $S$. This is clear because the proof of \Cref{lem:qRK_STD_noncorruptE} relies only on the size of $B\sm S$ and the fact that the indices of $B\sm S$ do not lie in the support of $\xi$. 

        \begin{lem}
            \label{lem:qRK_EH_differentindexE}
            Suppose $\beta < q <1-\beta$ and $\eta,\xi\in \R^m$ with $\supp(\eta)\cap \supp(\xi)=\varnothing$. Let $\x$ be the solution to the system $Ax=b_t$ and $\|\xi\|_0\leq \beta m$. Suppose
            \[
                \frac{q}{q-\beta}\groupp{\frac{\beta m}{\s_{\max}^2(A)}+\frac{2\beta}{1-q-\beta}+\frac{4\beta\sqrt{m}}{\s_{\max}(A)\sqrt{1-q-\beta}}}<\frac{\s_{q-\beta,\min}^2(A)}{\s_{\max}^2(A)}.
            \]
            If $x_{k+1}$ is the random variable defined by \eqref{eq:qRK_EH_xkplus1scheme}, then
            \[
                \E\|x_{k+1}-\x\|^2\leq (1-C)\E\|x_k-\x\|^2+\frac{2\beta(1-q)}{q(1-q-\beta)}\|\eta\|_\infty^2,
            \]
            where
            \[
                C=(q-\beta)\frac{\s_{q-\beta,\min}^2(A)}{q^2m}-\frac{\beta}{q}\groupp{1+\frac{2\s_{\max}^2(A)}{m(1-q-\beta)}+\frac{4\s_{\max}(A)}{\sqrt{m}\sqrt{1-q-\beta}}}.
            \]
        \end{lem}
        \begin{proof}
            This proof is similar to that of \Cref{cor:qRK_STD_NEWbound}, where the only difference stems from the addition of an error horizon term and the use of a slightly different quantile bound. Note
            \begin{align*}
            \E\|x_{k+1}-\x\|^2 &= \PP(i\in S)\E_{i\in S}\|x_{k+1}-\x\|^2+\PP(i\in B\sm S)\E_{i\in B\sm S}\|x_{k+1}-\x\|^2\\
            &\leq \frac{|S|}{qm}\groupp{2\groupp{1+\frac{\s_{\max}^2(A)}{m(1-q-\beta)}+\frac{2\s_{\max}(A)}{\sqrt{m}\sqrt{1-q-\beta}}}\|x_k-\x\|^2+\frac{2\|\eta\|_\infty^2}{1-q-\beta}}\\
            &\quad+\groupp{1-\frac{|S|}{qm}}\groupp{1-\frac{\s_{q-\beta,\min}^2(A)}{qm}}\|x_k-\x\|^2,
        \end{align*}
        which is increasing in $|S|$. By re-arranging after substituting $\beta m$ for $|S|$, we arrive at 
        \begin{equation}
            \begin{split}
                \E\|x_{k+1}-\x\|^2 &\leq \groupb{1+\frac{\beta}{q}\groupp{1+\frac{2\s_{\max}^2(A)}{m(1-q-\beta)}+\frac{4\s_{\max}(A)}{\sqrt{m}\sqrt{1-q-\beta}}}-(q-\beta)\frac{\s_{q-\beta,\min}^2(A)}{q^2m}}\|x_k-\x\|^2 \\
                &\quad+\frac{2\beta(1-q)}{q(1-q-\beta)}\|\eta\|_\infty^2. 
            \end{split}
            \label{eq:nonconditional_expectation}
        \end{equation}
        To ensure decay, we require 
        \[
            \frac{q}{q-\beta}\groupp{\frac{\beta m}{\s_{\max}^2(A)}+\frac{2\beta}{1-q-\beta}+\frac{4\beta\sqrt{m}}{\s_{\max}(A)\sqrt{1-q-\beta}}}<\frac{\s_{q-\beta,\min}^2(A)}{\s_{\max}^2(A)}.
        \]
        \end{proof}

        Now the expectation in \Cref{lem:qRK_EH_differentindexE} and the expectation $\E\|x_{k+1}-\x\|^2$, where $x_{k+1}$ is given by applying qRK on $~{Ax=b_t+\eta+\xi}$, are both taken over the same indices. This is exactly how the proof of \Cref{theorem:qRK_EH} proceeds:

        \begin{proof}[Proof of \Cref{theorem:qRK_EH}]
            Note that $x_k-\x=x_k\dl-\x+\eta a_i$ and $\|x_k-\x\|^2=\|x_k\dl-\x\|^2 + |\eta_i|^2$, where $x_k$ is the projection of $x_{k-1}$ onto $H_i\dl$ when hyperplane $H_i$ is chosen to produce $x_k$. Then, if we let $C$ assume the same value as in \Cref{lem:qRK_EH_differentindexE}, 
            \begin{align*}
                \E\|x_k-\x\|^2 &\leq \E\|x_k\dl-\x\|^2 + \|\eta\|_\infty^2 \\
                &\leq (1-C)\|x_{k-1}-\x\|^2 + \groupp{\frac{2\beta(1-q)}{q(1-q-\beta)}+ 1}\|\eta\|_\infty^2,
            \end{align*}
            whereby taking the total expectation and applying induction, we arrive at
            \begin{align*}
                \E\|x_k-\x\|^2 &\leq (1-C)^k\|x_0-\x\|^2+\groupp{\frac{2\beta(1-q)}{q(1-q-\beta)}+ 1}\|\eta\|_\infty^2\sum_{j=0}^{k-1}(1-C)^j \\
                &\leq (1-C)^k\|x_0-\x\|^2+\frac{1}{C}\groupp{\frac{2\beta(1-q)}{q(1-q-\beta)}+ 1}\|\eta\|_\infty^2.
            \end{align*}
        \end{proof}
        
    \subsection{Proof of the dqRK Error Horizon Bound (\Cref{theorem:dqRK_EH})}
        \label{sec:dqRK_EH_proof}
        Let $r_j=\lll x_k,a_j\rrr - (b_t-\xi)_j$, $Q_0=q_0\quant(\{|r_j-\eta_j|\}_{j=1}^m)$, and $Q=q\quant(\{|r_j-\eta_j|\}_{j=1}^m)$, i.e. $r_j$ is the $j$-th residual entry of $Ax_k-b_t-\xi$ and $Q_0$ and $Q$ are the quantiles determined by dqRK on the system \eqref{eq:fullcorrupsys}. Just as in \Cref{sec:qRK_EH_proof} we first bound $\E\|x_{k+1}-\x\|^2$ when 
        \begin{equation}
            x_{k+1}=x_k+((b_t+\xi)_i-\lll x_k,a_i\rrr)a_i\quad \text{for}\quad i\sim \unif(\{j\in [m]:Q_0<|r_j-\eta_j|\leq Q\}). \label{eq:dqRK_EH_xkplus1scheme}
        \end{equation}
        For the remainder of this section, $S=\{j\in C: Q_0<|r_j-\eta_j|\leq Q\}$ and $B= \{j\in [m]: Q_0<|r_j-\eta_j|\leq Q\}$ where $C$ is the indices of non-zero entries of $\xi$, i.e. $C=\supp(\xi)$. The results of this section leverage the work done for \Cref{theorem:qRK_EH}. Namely, \Cref{cor:dqRK_EH_quantilebound} and \Cref{cor:dqRK_EH_corruptE} are the same as \Cref{cor:qRK_EH_quantilebound} and \Cref{cor:qRK_EH_corruptE} (including the proofs), respectively, with $q$ in place of $q$. The analogue of \Cref{lem:qRK_STD_noncorruptE} and \Cref{lem:qRK_EH_differentindexE} in the dqRK setting are also provided in this section.

        \begin{cor}
            \label{cor:dqRK_EH_quantilebound}
            Let $0<q_0<q<1-\beta$ and suppose $\x$ is a solution to $A\x=b_t$. If $\|\xi\|_0\leq \beta m$, then 
            \[
                Q\leq \frac{\s_{\max}\|x_k-\x\|}{\sqrt{m}\sqrt{1-q-\beta}}+ \frac{\sqrt{1-q}\|\eta\|_\infty}{\sqrt{1-q-\beta}}.
            \]
        \end{cor}
    
        \begin{cor}
            \label{cor:dqRK_EH_corruptE}
            In the same setting as \Cref{cor:dqRK_EH_quantilebound}, if $x_{k+1}$ is the random variable defined by \eqref{eq:dqRK_EH_xkplus1scheme} and $\supp(\xi)\cap \supp(\eta)=\varnothing$, then 
            \[
                \E_{i\in S} \| x_{k+1}-\x\|^2 \leq 2\groupp{1+\frac{\s_{\max}^2(A)}{m(1-q-\beta)}+\frac{2\s_{\max}(A)}{\sqrt{m}\sqrt{1-q-\beta}}}\|x_k-\x\|^2+\frac{2(1-q)\|\eta\|_\infty^2}{1-q-\beta}
            \]
        \end{cor}        
        \begin{cor}
            \label{cor:dqRK_EH_noncorruptE}
            In the same setting as \Cref{cor:dqRK_EH_corruptE}, 
            \[
                \E_{i\in B\sm S}\|x_{k+1}-\x\|^2\leq \groupp{1-\frac{\s_{q-\beta,\min}^2(A)}{q m}}\|x_k-\x\|^2.
            \]
        \end{cor}
        Though we may desire to use the bound 
        \[
            \E_{i\in B\sm S}\|x_{k+1}-\x\|^2\leq \groupp{1-\frac{\s_{q-\beta,\min}^2(A)}{q m}-\frac{\s_{q_0-\beta,\min}^2(A)}{q_0q m^2}}\|x_k-\x\|^2,
        \]
        akin to that of \cite[Lemma 2]{dqRK}, this would require the additional hypothesis that $\lll x_k,a_j\rrr =(b_t-\xi)_j$ for some $j\in[m]$. This proves to be difficult to satisfy for a proof \Cref{theorem:dqRK_EH} mimicking that of \Cref{theorem:qRK_EH}. Fortunately, a simple application of \cite[Lemma 1]{dqRK} shows without this hypothesis, it is still true that 
        \[
            \E_{i\in B\sm S}\|x_{k+1}-\x\|^2\leq \E_{i\in B_0\sm S_0}\|x_{k+1}-\x\|^2,
        \]
        where $B_0 = \{j\in [m]: |r_j-\eta_j|\leq Q\}$ and $S_0=\{j\in C: |r_j-\eta_j|\leq Q\}$, so that $B_0\sm S_0$ is the set of non-corrupt admissible indices. By \cite[Lemma 3]{qRK}, we have 
        \[
            \E_{i\in B\sm S}\|x_{k+1}-\x\|^2\leq \groupp{1-\frac{\s_{q-\beta,\min}^2(A)}{qm}}\|x_k-\x\|^2.
        \]
        \begin{lem}
            \label{lem:dqRK_EH_differentindexE}
            Suppose $\beta < q_0<q <1-\beta$, $q-q_0>\beta$, and $\eta,\xi\in \R^m$ with $\supp(\eta)\cap \supp(\xi)=\varnothing$. Let $\x$ be the solution to the system $Ax=b_t$ and $\|\xi\|_0\leq \beta m$. Suppose
            \[
                \frac{q}{q-q_0-\beta}\groupp{\frac{\beta m}{\s_{\max}^2(A)}+\frac{2\beta}{1-q-\beta}+\frac{4\beta\sqrt{m}}{\s_{\max}(A)\sqrt{1-q-\beta}}}<\frac{\s_{q-\beta,\min}^2(A)}{\s_{\max}^2(A)}.
            \]
            If $x_{k+1}$ is the random variable defined by \eqref{eq:dqRK_EH_xkplus1scheme}, with $\lll x_0,a_j\rrr=(b_t+\xi)_j$ for some $j\in [m]$, then
            \[
                \E\|x_{k+1}-\x\|^2\leq (1-C)\E\|x_k-\x\|^2+\frac{2\beta(1-q)}{(q-q_0)(1-q-\beta)}\|\eta\|_\infty^2,
            \]
            where
            \[
                C=(q-q_0-\beta)\frac{\s_{q-\beta,\min}^2(A)}{(q-q_0)q m}-\frac{\beta}{q-q_0}\groupp{1+\frac{2\s_{\max}^2(A)}{m(1-q-\beta)}+\frac{4\s_{\max}(A)}{\sqrt{m}\sqrt{1-q-\beta}}}.
            \]
        \end{lem}
        \begin{proof}
            This proof is almost identical to \Cref{lem:qRK_EH_differentindexE}, except for the different probabilities $\PP(i\in S)$ and $\PP(i\not\in S)$. Note from \Cref{cor:dqRK_EH_corruptE} and \Cref{cor:dqRK_EH_noncorruptE}, 
            \begin{align*}
            \E\|x_{k+1}-\x\|^2 &= \PP(i\in S)\E_{i\in S}\|x_{k+1}-\x\|^2+\PP(i\in B\sm S)\E_{i\in B\sm S}\|x_{k+1}-\x\|^2\\
            &\leq \frac{|S|}{(q-q_0)m}\groupp{2\groupp{1+\frac{\s_{\max}^2(A)}{m(1-q-\beta)}+\frac{2\s_{\max}(A)}{\sqrt{m}\sqrt{1-q-\beta}}}\|x_k-\x\|^2+\frac{2\|\eta\|_\infty^2}{1-q-\beta}}\\
            &\quad+\groupp{1-\frac{|S|}{(q-q_0)m}}\groupp{1-\frac{\s_{q-\beta,\min}^2(A)}{q m}}\|x_k-\x\|^2,
        \end{align*}
        which is increasing in $|S|$. By noting $|S|\leq \beta m$ and re-writing this expression, it is clear 
        \begin{align*}
            \E\|e_{k+1}\|^2 &\leq \Bigg[1+\frac{\beta}{q-q_0}\groupp{1+\frac{2\s_{\max}^2(A)}{m(1-q-\beta)}+\frac{4\s_{\max}(A)}{\sqrt{m}\sqrt{1-q-\beta}}}-(q-q_0-\beta)\frac{\s_{q-\beta,\min}^2(A)}{(q-q_0)q m}\Bigg]\|e_k\|^2 \\
            &\quad+\frac{2\beta(1-q)}{(q-q_0)(1-q-\beta)}\|\eta\|_\infty^2,
        \end{align*}
        where $e_{k}=x_k-\x$. To ensure decay, we require 
        \[
            \frac{q}{q-q_0-\beta}\groupp{\frac{\beta m}{\s_{\max}^2(A)}+\frac{2\beta}{1-q-\beta}+\frac{4\beta\sqrt{m}}{\s_{\max}(A)\sqrt{1-q-\beta}}}<\frac{\s_{q-\beta,\min}^2(A)}{\s_{\max}^2(A)}.
        \]
        \end{proof}

        The proof of \Cref{theorem:dqRK_EH} follows the same argument used to prove \Cref{theorem:qRK_EH} with \Cref{lem:dqRK_EH_differentindexE} used in place of \Cref{lem:qRK_EH_differentindexE}.

\section{Acknowledgements}
This work is partially supported by the 2024-2025 Rose Hills Innovator Grant. We would also like to thank Elizaveta Rebrova for helping improve earlier versions of this work.

\bibliographystyle{plain}
\bibliography{refs}

\begin{thebibliography}{10}

\bibitem{dqRK}
Emeric Battaglia and Anna Ma.
\newblock {R}everse {Q}uantile-{RK} and its {A}pplication to {Q}uantile-{RK}, 2024.

\bibitem{timevaryingcorruptionqRK}
Nestor Coria, Jamie Haddock, and Jaime Pacheco.
\newblock On quantile randomized kaczmarz for linear systems with time-varying noise and corruption, 2024.

\bibitem{greedworks}
Jamie Haddock and Anna Ma.
\newblock {G}reed {W}orks: {A}n {I}mproved {A}nalysis of {S}ampling {K}aczmarz--{M}otzkin.
\newblock {\em SIAM Journal on Mathematics of Data Science}, 3(1):342--368, 2021.

\bibitem{qRK2}
Jamie Haddock, Deanna Needell, Elizaveta Rebrova, and William Swartworth.
\newblock {Q}uantile-{B}ased {I}terative {M}ethods for {C}orrupted {S}ystems of {L}inear {E}quations.
\newblock {\em SIAM Journal on Matrix Analysis and Applications}, 43(2):605--637, 2022.

\bibitem{adversarialcorruption}
Longxiu Huang, Xia Li, and Deanna Needell.
\newblock Randomized kaczmarz in adversarial distributed setting.
\newblock {\em SIAM Journal on Scientific Computing}, 46(3):B354--B376, 2024.

\bibitem{qRK_EH_subgaussian}
Benjamin Jarman and Deanna Needell.
\newblock Quantilerk: Solving large-scale linear systems with corrupted, noisy data, 2021.

\bibitem{kacz}
S.~Kaczmarz.
\newblock Angen\"aherte {A}ufl\"osung von {S}ystemen linearer {G}leichungen.
\newblock {\em Bull. Internat. Acad. Polon.Sci. Lettres A}, pages 335--357, 1937.

\bibitem{needell_LS}
Deanna Needell.
\newblock Randomized {K}aczmarz solver for noisy linear systems.
\newblock {\em BIT}, 50(2):395–403, jun 2010.

\bibitem{SGD_RK}
Deanna Needell, Nathan Srebro, and Rachel Ward.
\newblock {S}tochastic gradient descent, weighted sampling, and the randomized {K}aczmarz algorithm.
\newblock {\em Mathematical Programming}, 155(1):549--573, Jan 2016.

\bibitem{block_LS}
Deanna Needell, Ran Zhao, and Anastasios Zouzias.
\newblock Randomized block kaczmarz method with projection for solving least squares.
\newblock {\em Linear Algebra and its Applications}, 484:322--343, 2015.

\bibitem{sparseREK}
Frank Sch\"opfer, Dirk~A. Lorenz, Lionel Tondji, and Maximilian Winkler.
\newblock Extended randomized {K}aczmarz method for sparse least squares and impulsive noise problems.
\newblock {\em Linear Algebra and its Applications}, 652:132--154, 2022.

\bibitem{lowestlossSGD}
Vatsal Shah, Xiaoxia Wu, and Sujay Sanghavi.
\newblock Choosing the sample with lowest loss makes sgd robust, 2020.

\bibitem{qRK}
Stefan Steinerberger.
\newblock {Quantile-based Random Kaczmarz for corrupted linear systems of equations}.
\newblock {\em Information and Inference: A Journal of the IMA}, 12(1):448--465, 02 2022.

\bibitem{RK}
Thomas Strohmer and Roman Vershynin.
\newblock A {R}andomized {K}aczmarz {A}lgorithm with {E}xponential {C}onvergence.
\newblock {\em Journal of Fourier Analysis and Applications}, 15, 03 2007.

\bibitem{sparse_qRK_EH}
Lu~Zhang, Hongxia Wang, and Hui Zhang.
\newblock Quantile-based random sparse kaczmarz for corrupted and noisy linear systems.
\newblock {\em Numerical Algorithms}, 98(3):1575--1610, March 2025.

\bibitem{Zouzias_REK}
Anastasios Zouzias and Nikolaos~M. Freris.
\newblock Randomized {E}xtended {K}aczmarz for {S}olving {L}east {S}quares.
\newblock {\em SIAM Journal on Matrix Analysis and Applications}, 34(2):773--793, 2013.

\end{thebibliography}

\end{document}